\documentclass[11pt,a4paper]{amsart}

\usepackage{color}
\usepackage{amsmath, enumerate}
\usepackage{amsfonts}
\usepackage{amsthm}
\usepackage{mathrsfs}
\usepackage{amssymb}
\usepackage[english]{babel}
\usepackage{graphicx}
\usepackage[all]{xy}
\usepackage{yfonts}

\usepackage[active]{srcltx}
\usepackage[parfill]{parskip}

\newtheorem{theorem}{Theorem}[section]
\newtheorem{lemma}[theorem]{Lemma}
\newtheorem{definition}[theorem]{Definition}
\newtheorem{corollary}[theorem]{Corollary}
\newtheorem{proposition}[theorem]{Proposition}
\newtheorem{remark}[theorem]{Remark}

\newcommand{\Ob}{{\mathrm{Ob}}}
\newcommand{\Hom}{{\mathrm{Hom}}}
\newcommand{\Ext}{{\mathrm{Ext}}}

\newcommand{\add}{{\mathrm{add}}}

\newcommand{\eins}{\leavevmode\hbox{\small1\kern-3.8pt\normalsize1}}

\newcommand{\res}{{\rm Res} }

\newcommand{\ind}{{\rm Ind} }

\newcommand{\tto}{\twoheadrightarrow}

\newcommand{\mC}{\mathbb{C}}
\newcommand{\mN}{\mathbb{N}}

\newcommand{\mZ}{\mathbb{Z}}

\newcommand{\mV}{\mathbb{V}}

\newcommand{\dd}{{\mathbf{d}}}

\newcommand{\fa}{{\mathfrak a}}
\newcommand{\fb}{{\mathfrak b}}

\newcommand{\fg}{{\mathfrak g}}
\newcommand{\fh}{{\mathfrak h}}

\newcommand{\fl}{{\mathfrak l}}
\newcommand{\fn}{{\mathfrak n}}

\newcommand{\fq}{{\mathfrak q}}

\newcommand{\fu}{{\mathfrak u}}

\newcommand{\fw}{{\mathfrak w}}

\newcommand{\cD}{\mathcal{D}}
\newcommand{\cM}{\mathcal{M}}

\newcommand{\cP}{\mathcal{P}}
\newcommand{\cQ}{\mathcal{Q}}

\newcommand{\cR}{\mathcal{R}}
\newcommand{\cK}{\mathcal{K}}
\newcommand{\cS}{\mathcal{S}}

\newcommand{\cO}{\mathcal{O}}

\newcommand{\cA}{\mathcal{A}}

\newcommand{\cL}{\mathcal{L}}

\newcommand{\cZ}{\mathcal{Z}}
\newcommand{\cY}{\mathcal{Y}}
\newcommand{\cX}{\mathcal{X}}

\newcommand{\cT}{\mathcal{T}}
\newcommand{\intdom}{\Lambda^+_{\mathrm{int}}}

\newcommand{\End}{{\rm End}}
\newcommand{\Id}{{\rm Id}}

\newtheorem{innercustomthm}{Theorem}
\newenvironment{customthm}[1]
  {\renewcommand\theinnercustomthm{#1}\innercustomthm}
  {\endinnercustomthm}

\begin{document}

\title[Dualities and derived equivalences for category~$\cO$]
{Dualities and derived equivalences\\ for category~$\cO$}

\author{Kevin Coulembier and Volodymyr Mazorchuk}
\date{}

\begin{abstract}
We determine the Ringel duals for all blocks in the parabolic versions 
of the BGG category~$\cO$ associated to a reductive finite dimensional 
Lie algebra. In particular we find that, contrary to the original 
category~$\cO$ and the specific previously known cases in the parabolic 
setting, the blocks are not necessarily Ringel self-dual. However, 
the parabolic category~$\cO$ as a whole is still Ringel self-dual. 
Furthermore, we use generalisations of the Ringel duality functor to 
obtain large classes of derived equivalences between blocks in parabolic and original 
category~$\cO$.
We subsequently classify all derived equivalence classes of blocks 
of category~$\cO$ in type~$A$ which preserve the Koszul  grading. 
\end{abstract}

\maketitle

\noindent
\textbf{MSC 2010 : }  16E35, 16E30, 17B10, 16S37

\noindent
\textbf{Keywords :} Ringel duality, Koszul duality, equivalences of derived categories, BGG category~$\cO$

\section{Introduction and main results}
For a reductive finite dimensional complex Lie algebra $\fg$ with 
a fixed triangular decomposition, consider the Bernstein-Gelfand-Gelfand
category~$\cO$ from \cite{BGG} (see also \cite{Humphreys}) and its parabolic generalisation 
by Rocha-Caridi in~\cite{RC}. It is well-known that blocks in parabolic 
category~$\cO$ are described  by quasi-hereditary algebras which possess a Koszul grading, 
see~\cite{Back, BGS, SoergelD, MR2563180}. Moreover, graded lifts of blocks 
of category~$\cO$ are Koszul dual to graded lifts of blocks in the parabolic 
versions of~$\cO$. It was proved in~\cite{SoergelT} that parabolic category $\cO$, 
as a whole, is Ringel self-dual. It is well-known that for full category~$\cO$, 
even every individual block is Ringel self-dual, as follows from \cite[Struktursatz~9]{SoergelD}.
The same holds for the {\em principal}
block in parabolic category~$\cO$, see~\cite{prinjective}. One of our observations which motivated the present paper is that arbitrary
blocks in parabolic category~$\cO$ are, surprisingly, {\em not} Ringel self-dual in 
general. This is our first main result.

\begin{customthm}{A}
Blocks in parabolic category~$\cO$ are, in general, not Ringel self-dual. 
However, every block is Ringel dual to another block 
in the same parabolic category. The Ringel dual of a block is also equivalent to a block with the same 
central character in a different parabolic version of~$\cO$.
\end{customthm}

Ringel duality between two blocks always implies an equivalence, as 
triangulated categories, of the bounded derived categories of these 
blocks given by the derived {\em Ringel duality functor}, 
see~\cite{Ringel, Happel, Rickard}. From our proof of Theorem~A, it follows 
easily that, for parabolic category~$\cO$, this equivalence lifts to the 
graded setting for the Koszul grading. On the other hand, the {\em Koszul 
duality functor}, see~\cite{BGS, MOS}, also induces an equivalence between 
the derived categories of a graded lift of a block and its Koszul dual. 
However, this equivalence does not correspond to a derived equivalence in 
the ungraded sense. Moreover, we show explicitly that two Koszul dual algebras need 
not be derived equivalent in the sense of \cite{Rickard}, by considering 
an example for category~$\cO$.

We are interested in derived equivalences in the sense of \cite{Rickard} 
which lift to the graded setting, see explicit Definition \ref{defgradeq}. 
We refer to this as a {\em gradable derived equivalence} or being 
{\em gradable derived equivalent}. Hence, for parabolic category~$\cO$, the Ringel 
duality functor induces a gradable derived equivalence, whereas the Koszul 
duality functor does not. 

In the current paper we start the systematic
study of derived equivalences for (parabolic)
category $\cO$.
Some interesting derived equivalences for $\fg=\mathfrak{sl}(n)$ have been obtained by Chuang and 
Rouquier in \cite{sl2} and by Khovanov in \cite{Khovanov}. By construction, 
the ones in \cite{sl2} are gradable. We take an approach independent from the previous
results but recover the derived equivalences in \cite{sl2, Khovanov}.
Our results are rather conclusive
for category $\cO$ for type $A$, whereas there remain open questions for the parabolic 
versions of $\cO$ over other reductive Lie algebras.

Consider a reductive finite dimensional complex 
Lie algebra $\fg$ with a fixed Cartan subalgebra $\fh$
and a fixed Borel subalgebra $\fb$ containing $\fh$. Consider $\cO=\cO(\fg,\fb)$
and let $W=W(\fg:\fh)$ be the
Weyl group and $\Lambda_{\mathrm{int}}$ the set of integral weights. 
For a coset $\Lambda\in\fh^\ast/\Lambda_{\mathrm{int}}$, we denote the corresponding integral 
Weyl group by $W_{\Lambda}$. For a dominant $\lambda\in\Lambda$,
the stabiliser of $\lambda$ in $W_\Lambda$ under the dot action is denoted 
by $W_{\Lambda,\lambda}$ and the block in the category~$\cO(\fg,\fb)$ containing 
the simple highest weight module with highest weight $\lambda$ 
by $\cO_\lambda(\fg,\fb)$. Our second 
main result is an analogue of \cite[Theorem~11]{SoergelD} for 
derived categories, restricted to type~$A$.

\begin{customthm}{B}
Consider two Lie algebras $\fg$ and $\fg'$ of type~$A$, with respective Borel 
subalgebras $\fb$ and $\fb'$. Then there is a gradable derived equivalence
\begin{displaymath}
\cD^b(\cO_\lambda(\fg,\fb))\;\cong\;\cD^b(\cO_{\lambda'}(\fg',\fb')) 
\end{displaymath}
for dominant $\lambda\in\Lambda$ and $\lambda'\in\Lambda'$, if and only if, 
for some decompositions 
\begin{displaymath}
W_{\Lambda}\cong X_1\times X_2\times \dots\times X_k\quad\text{ and }\quad
W_{\Lambda'}\cong X'_1\times X'_2\times \dots\times X'_m
\end{displaymath}
into products of irreducible Weyl groups, we have $k=m$ and there is a permutation
$\varphi$ on $\{1,2,\dots,k\}$ such that $W_{\Lambda,\lambda}\cap X_i\cong
W_{\Lambda',\lambda'}\cap X'_{\varphi(i)}$ and $X_i\cong X'_{\varphi(i)}$, for all $i=1,2,\dots,k$.
\end{customthm}

Note that, according to \cite[Theorem~11]{SoergelD} (restricted to type $A$), there is an equivalence 
$\cO_\lambda(\fg,\fb)\cong \cO_{\lambda'}(\fg',\fb')$ if the 
following stronger condition is satisfied: There is a Coxeter group isomorphism $W_{\Lambda}\to W_{\Lambda'}$ which swaps $W_{\Lambda,\lambda}$ and $W_{\Lambda',\lambda'}$.

To present the remainder of the main results, we need more notation. 
We consider again an arbitrary complex reductive Lie algebras $\fg$ and henceforth 
only consider integral weights, which is justified by the results in \cite{SoergelD}. Hence $\Lambda=\Lambda_{\mathrm{int}}$ and we leave out the reference to $\Lambda$. For every integral dominant 
weight $\lambda$, the block $\cO_\lambda$ now consists of all modules in $\cO$ 
with the same generalised central character as the simple module with highest 
weight $\lambda$. To any 
integral dominant weight $\mu$, we associate a parabolic subalgebra $\fq_{\mu}$ 
of~$\fg$, uniquely defined by the fact that the Weyl group of its Levi factor is $W_\mu$. The full 
subcategory of $\cO$ consisting of all modules which are locally $\fq_{\mu}$-finite 
is denoted by $\cO^\mu$. The integral part of this category decomposes 
naturally into subcategories $\cO_\lambda^\mu$. Graded lifts of the latter categories, with respect 
to the Koszul grading, are denoted by~${}^{\mZ}\cO_\lambda^\mu$.

\begin{customthm}{C}
Consider the algebra $\fg=\mathfrak{sl}(n)$ and four integral dominant weights 
$\lambda,\lambda',\mu,\mu'$. If we have isomorphisms of groups 
\begin{displaymath}
W_\lambda\cong W_{\lambda'}\qquad\mbox{and}\qquad W_\mu\cong W_{\mu'}, 
\end{displaymath} 
then there is a gradable derived equivalence between $\cO_\lambda^\mu$ and 
$\cO_{\lambda'}^{\mu'}$. In particular, we have equivalences of triangulated 
categories 
\begin{displaymath}
\cD^b(\cO^\mu_\lambda)\;\cong\; \cD^b(\cO^{\mu'}_{\lambda'})
\qquad\mbox{and}\qquad \cD^b({}^{\mZ}\cO^\mu_\lambda)\;\cong\; \cD^b({}^{\mZ}\cO^{\mu'}_{\lambda'}). 
\end{displaymath}
\end{customthm}

A more complicated formulation of this result for arbitrary $\fg$
can be found in Theorem~\ref{maindereq}. Theorem~C generalises 
\cite[Proposition~7]{Khovanov}, which corresponds precisely to the case $\lambda=\lambda'=0$ 
in the ungraded setting. Furthermore, our approach gives an explicit form of the 
functor and the tilting complex, which describe the equivalence. It also provides a purely 
algebraic proof, whereas the proof in~\cite{Khovanov} depends on the geometric 
description of  $\cO$. Theorem~C, for $\mu=\mu'$, gives the derived equivalences
which can be constructed from \cite[Theorem~6.4 and Section~7.4]{sl2}. It seems that
in this case
even the functors describing the derived equivalence are isomorphic, as can be checked by hand
for small cases. In \cite{sl2}
these functors have the elegant property that they are defined directly to act between
the two relevant categories. On the other hand, they are defined in terms of total
complexes for a complex of functors which becomes arbitrarily big. The functors
in the current paper have the drawback that they are defined implicitly by using
an auxiliary regular block in category $\cO$, see Theorem D. The advantage is that the functor
on the latter category inducing the equivalence is a well-understood functor with
elegant properties.

Our results are formulated in terms of four types of functors on category~$\cO$, 
{\it viz.} projective, Zuckerman, twisting and shuffling functors, see~\cite{simple} 
or the preliminaries for an overview. Twisting functors commute with 
projective functors, see \cite{AS}, and shuffling  functors commute with 
Zuckerman functors, see \cite{simple}. The non-trivial commutation relations 
between twisting and Zuckerman functors, and between shuffling and projective functors, 
lie at the origin of the failure of blocks to be Ringel self-dual. It is also 
these commutation relation that we exploit to obtain the derived equivalences. 
As an extra result, we extend the Koszul duality 
between projective and Zuckerman functors of 
\cite{RH} to the full parabolic and singular setting. 

For any  $x\in W$, the derived twisting functor $\cL T_x$ and shuffling functor $\cL C_x$ 
are auto-equivalences of the category $\cD^b(\cO_0)$. For an integral dominant $\nu$, we introduce 
the notation $W_\nu^\dagger$ for the subgroup of~$W$ generated by all simple 
reflections which are orthogonal to all reflections in $W_\nu$.
Let $w_0^\nu$ stand for the longest element of~$W_\nu$. 
Theorem~C can be obtained by an 
iterative application of the following theorem.

\begin{customthm}{D}
Consider $\fg$ a finite dimensional complex reductive Lie algebra and an integral 
dominant~$\nu,\lambda,\mu\in\fh^\ast$ such that $W_\nu$ is of type~$A$.
\begin{enumerate}[$($i$)$]
\item Assume that 
$W_\mu\subset W_\nu\times W_\nu^\dagger$. Then there is a dominant integral 
$\mu'$ with $W_{\mu'}=w_0^\nu W_\mu w_0^\nu$. The auto-equivalence 
$\cL T_{w_0^\nu}$ of $\cD^b(\cO_0)$ restricts to an equivalence of triangulated 
categories between two subcategories, equivalent to, respectively, $\cD^b(\cO^\mu_\lambda)$ 
and $ \cD^b(\cO^{\mu'}_\lambda)$, yielding a gradable derived equivalence 
$\cD^b(\cO^\mu_\lambda)\;\tilde\to\; \cD^b(\cO^{\mu'}_\lambda)$.
\item\label{may26-1} Assume that 
$W_\lambda\subset W_\nu\times W_\nu^\dagger$. Then there is a dominant integral 
$\lambda'$ with $W_{\lambda'}=w_0^\nu W_\lambda w_0^\nu$. The 
auto-equivalence $\cL C_{w_0^\nu}$ of $\cD^b(\cO_0)$ restricts 
to an equivalence of triangulated categories between two subcategories, 
equivalent to, respectively, $\cD^b(\cO^\mu_\lambda)$ and $ \cD^b(\cO^{\mu}_{\lambda'})$, 
yielding a gradable derived equivalence $\cD^b(\cO^\mu_\lambda)\;\tilde\to\; \cD^b(\cO^{\mu}_{\lambda'})$.\end{enumerate}
\end{customthm}

Note that the two parts can be interpreted as Koszul duals of one another 
using \cite[Section~6.5]{MOS}. When $w_0^\nu=w_0$ and $\mu=0$ in 
Theorem~D\eqref{may26-1}, the categories $\cO_\lambda$ and $\cO_{\lambda'}$ are equivalent 
by \cite[Theorem~11]{SoergelD}. However, our equivalence of the derived categories 
is not induced by that equivalence, but is rather
given by the derived Ringel duality functor. 

The paper is organised as follows. In Section~\ref{secprel} we recall some 
results on category~$\cO$ and Koszul and Ringel duality. In Section~\ref{secgranongra} 
we give an explicit example of Koszul dual algebras which are not derived 
equivalent and introduce the notion of gradable derived equivalence. We use this to study shuffling in the parabolic setting. In 
Section~\ref{secGradTrans} we obtain several results on the graded lifts of 
translation functors. In particular, we extend 
the Koszul duality of \cite{RH} between translation functors and parabolic 
Zuckerman functors to the generality we will need it. In 
Section~\ref{secProjShuff} we study the commutation relations 
between shuffling and projective functors. In 
Section~\ref{secConstr} we construct the derived equivalences between 
blocks in parabolic category~$\cO$, proving Theorems~C and~D. This is 
used in Section~\ref{secClass} to classify the blocks in category~$\cO$ 
for Lie algebras of type~$A$ up to gradable derived equivalence, proving 
Theorem~B. In Section~\ref{secRingel} we determine the Ringel duals of all 
blocks in parabolic category~$\cO$, proving Theorem~A, and study the Koszul-Ringel duality functor.


\section{Preliminaries}\label{secprel}

We work over $\mathbb{C}$. Unless explicitly stated otherwise,
commuting diagrams of functors commute only up to a natural isomorphism. 
{\em Graded} always refers to {\em $\mathbb{Z}$-graded}.

\subsection{Category~$\cO$ and its parabolic generalisations}\label{prel1}
We consider the BGG category~$\cO$, associated to a triangular decomposition of a 
finite dimensional complex reductive Lie algebra
$\fg=\fn^-\oplus\fh\oplus\fn^+$, see~\cite{BGG, Humphreys}. For any weight 
$\nu\in\fh^\ast$, we denote the corresponding simple highest weight module by $L(\nu)$. We also introduce an involution on~$\fh^\ast$ by setting 
$\widehat\nu=-w_0(\nu)$, with $w_0$ the longest element of the Weyl group 
$W=W(\fg:\fh)$. We denote by $\langle\cdot,\cdot\rangle$ a $W$-invariant inner product 
on $\fh^\ast$ and the set of integral, not necessarily regular, dominant weights by~$\intdom$.
For any $\lambda\in\intdom$, the indecomposable block in category~$\cO$ containing $L(\lambda)$ is 
denoted by~$\cO_\lambda$. 

For $B$ the set of simple positive roots and $\mu\in\intdom$, 
set $B_\mu=\{\alpha\in B\,|\, \langle\mu+\rho,\alpha\rangle=0\}$. Let 
$\fu^{-}_\mu$ be the subalgebra of~$\fg$ generated by the root spaces corresponding 
to the roots in~$-B_\mu$. Then we have the parabolic subalgebra $\fq_\mu$ of~$\fg$, 
given by
\begin{displaymath}
\fq_\mu:=\fu_{\mu}^-\oplus\fh\oplus\fn^+.
\end{displaymath}

The full subcategory of~$\cO_\lambda$ with objects given by the modules 
in~$\cO_\lambda$ which are $U(\fq_\mu)$-locally finite is denoted 
by~$\cO_\lambda^\mu$. These are subcategories of parabolic category~$\cO$ as 
introduced in~\cite{RC}. By construction, $\cO^\mu_\lambda$ is a Serre 
subcategory of~$\cO_\lambda$. We denote the corresponding exact full 
embedding of categories by $\mathbf{\imath}^\mu:\cO^\mu\hookrightarrow \cO$. 
The left adjoint of $\mathbf{\imath}^\mu$ is the corresponding {\em Zuckerman functor},
denoted by $Z^{\mu}$, it is given by taking the largest quotient inside $\cO^\mu$.
The categories $\cO_0^\mu$ and $\cO_\lambda$ are indecomposable. The category 
$\cO_\lambda^\mu$ may decompose, e.g. in the case  $\mathfrak{g}=B(2)=\mathfrak{so}(5)$, 
for $w_0^\lambda=s$ and $w_0^\mu=t$, where $s$ and $t$ are the two different simple reflections. 

We define the set $X_\lambda$ as the set of {\em longest} representatives in~$W$ of
cosets in $W/W_\lambda$. The non-isomorphic simple objects in the 
category~$\cO_\lambda$ are indexed by $X_\lambda$: 
\begin{displaymath}
\{L(w\cdot\lambda)\,|\,w\in X_\lambda\}.
\end{displaymath}
Now, for~$x\in X_\lambda$, the module $L(x\cdot\lambda)$ is an object of~$\cO^\mu_\lambda$ 
if and only if $x$ is a {\em shortest} representative in~$W$ of a coset in $W_\mu\backslash W$. 
The set of such shortest representatives $x\in X_\lambda$ is denoted by~$X_\lambda^\mu$. 

We denote by $\mathbf{d}$ the usual duality on $\cO$, which restricts to $\cO^\mu$ and to each block
in these categories, see \cite[Section~3.2]{Humphreys}.
For $x\in X_\lambda^\mu$, consider the following structural modules in $\cO^\mu_\lambda$: the standard module (or generalised Verma module) $\Delta^\mu(x\cdot\lambda)$ with simple top $L(x\cdot\lambda)$, the costandard module $\nabla^\mu(x\cdot\lambda):=\dd \Delta^\mu(x\cdot\lambda)$, the injective envelope $I^\mu(x\cdot\lambda)$ and projective cover $P^\mu(x\cdot\lambda)$ of~$L(x\cdot\lambda)$, the indecomposable quasi-hereditary tilting module $T^\mu(x\cdot\lambda)$
with highest weight $x\cdot\lambda$.

Consider a minimal projective generator of~$\cO^\mu_\lambda$ given by 
\begin{equation}\label{progen} P^\mu_\lambda\;:=\;\bigoplus_{x\in X_\lambda^\mu}P^{\mu}(x\cdot\lambda),
\end{equation}
where $P^{\mu}(x\cdot\lambda)$ is the indecomposable projective cover of $L(x\cdot\lambda)$ in $\cO^\mu_\lambda$,
and the algebra $A_\lambda^\mu:=\End_{\fg}(P_\lambda^\mu)$. Then we have the usual 
equivalence  of categories
\begin{displaymath}
\cO^\mu_\lambda\,\;\tilde\to\;\,\mbox{mod-}A_\lambda^\mu;\qquad M\mapsto \Hom_{\fg}(P^\mu_\lambda,M). 
\end{displaymath}

We consider the Bruhat order $\le$ on~$W$, with the convention that $e$ is the smallest element. It restricts to the Bruhat 
order on~$X_\lambda^\mu$.
The order on the weights is defined by $x\cdot\lambda\le y\cdot\lambda$ if and only if $y\le x$. 
From the BGG Theorem on the structure of Verma modules, see e.g. \cite[Section~5.1]{Humphreys}, 
it follows that the algebras 
$A^\mu_\lambda$ are quasi-hereditary with respect to the poset of 
weights~$X_\lambda^\mu\cdot\lambda$. The standard modules coincide with the ones above.

\begin{remark}{\rm
We will use the term {\em (generalised) tilting module} 
for a module of a finite dimensional algebra satisfying properties (i)-(iii) in 
\cite[Section~III.3]{Happel}. When we refer to the 
modules in (parabolic) category~$\cO$ that simultaneously admit a standard and 
a costandard filtration, see e.g. \cite[Chapter~11]{Humphreys}, we will use 
the term {\em quasi-hereditary tilting module}, or q.h. tilting module. 
Then a q.h. tilting module is a (special case 
of) a {\em partial} generalised tilting module, whereas the {\em characteristic} 
q.h. tilting module is a generalised tilting module, see e.g. \cite[Theorem~5]{Ringel}.
} 
\end{remark}

When $\mu$ is regular, meaning that the corresponding parabolic category~$\cO^\mu$ is the usual category 
$\cO$, we  leave out the reference to $\mu$. Similarly, we will leave out $\lambda$ (as in $L(x):=L(x.\lambda)$), or replace it by $0$, 
whenever it is regular. By application of~\cite[ Theorem~11]{SoergelD}, all 
categories $\cO^\mu_\lambda$, with $\lambda$ arbitrary integral regular dominant 
and $\mu$ fixed, are equivalent, justifying this convention.

Consider the translation functor $\theta^{on}_\lambda:\cO_0\to\cO_\lambda$ to the $\lambda$-wall
and also its adjoint $\theta^{out}_\lambda:\cO_\lambda\to \cO_0$, which is the translation
out of the $\lambda$-wall, see \cite[Chapter~7]{Humphreys}. For $x\in W$, denote by $\theta_x$ 
the unique projective functor on~$\cO_0$ which maps $P(e)$ to $P(x)$, see~\cite{BG}. Note that, in particular,  $\theta^{out}_\lambda\circ\theta^{on}_\lambda=\theta_{w_0^\lambda}$. 
By \cite[Theorem~7.9]{Humphreys} or \cite{Jantzen}, for any $x\in W$, we have
\begin{equation}\label{thetaon}
\theta_\lambda^{on}L(x)=
\begin{cases}
L(x\cdot\lambda), & x\in X_\lambda;\\
0, &\text{otherwise}.
\end{cases}
\end{equation}

\subsection{Koszul duality for category~$\cO$}
Consider a finite dimensional algebra $B$. If $B$ is a quadratic positively 
graded algebra, we denote its {\em quadratic dual} by~$B^{!}$, as 
in \cite[Definition~2.8.1]{BGS}. If $B$ is, moreover, Koszul, we denote 
its {\em Koszul dual} by $E(B)=\Ext^{\bullet}_B(B_0,B_0)$. By Theorem~2.10.1 
in~\cite{BGS}, we have $E(B)=(B^{!})^{{\rm opp}}$ for any Koszul algebra~$B$. 
For a positively graded algebra $B$, we denote by $B$-gmod its category of 
finitely generated graded modules. 

For a complex $\cM^\bullet$ of graded modules, that is
$\displaystyle \cM^j=\bigoplus_{i\in \mathbb{Z}}\cM^j_i$, where $j\in\mathbb{Z}$,
we use the following convention
\begin{displaymath}
\left(\cM^\bullet [a]\langle b\rangle\right)^i_j=\cM^{i+a}_{j-b},  
\end{displaymath}
for shift in position in the complex and degree in the modules. This corresponds 
to the conventions in~\cite{BGS}, but differs slightly from~\cite{MOS}. 
A module $M$, regarded as an object in the derived category put in position zero, 
is denoted by~$M^\bullet$.

For any Koszul algebra $B$, \cite[Theorem~2.12.6]{BGS} introduces the 
{\em Koszul duality functor} $\cK_B$, a covariant equivalence of 
triangulated categories
\begin{equation}\label{Koszuleq}
\cK_B\,:\,\, \cD^b(B\mbox{-gmod})\,\;\tilde\to\,\; \cD^b(B^{!}\mbox{-gmod}).
\end{equation}

As in \cite[Section~3]{MOS} or \cite[Section~2]{Martinez}, we introduce the 
full subcategory of~$\cD^b(B\mbox{-gmod})$ of linear complexes of projective 
modules in~$B$-gmod, which we denote by $\mathfrak{LP}_B$. Then 
\cite[Theorem~2.4]{Martinez}  (or, more generally, \cite[Theorem~12]{MOS})
establishes, for any quadratic algebra $B$, an equivalence 
\begin{displaymath}
\epsilon_B:\mathfrak{LP}_B\;\,\tilde\to\;\, B^{!}\mbox{-gmod}.
\end{displaymath}
From  \cite[Chapter~5]{MOS} it follows that $\epsilon_B$ is isomorphic 
to the restriction of~$\cK_{B}$ in case $B$ is Koszul. 

As proved in~\cite{Back}, $A^\mu_\lambda$ has a Koszul grading, where the Koszul 
dual algebra is $E(A_\lambda^\mu)\cong A_{\widehat\mu}^\lambda$, see also 
\cite{BGS, SoergelD, MR2563180}. We also have 
$(A_\lambda^\mu)^{{\rm opp}}\cong A_\lambda^\mu$ (as graded algebras) as an 
immediate consequence of the duality functor~$\dd$. The algebras 
$A^\mu_\lambda$ are even standard Koszul in the sense of \cite{ADL}, see \cite{MR2563180}.

The graded module categories are denoted by ${}^{\mZ}\cO_\lambda^\mu:=A_\lambda^\mu$-gmod. 
We will sometimes replace the notation~$\Hom_{{}^\mZ\cO}$ by $\hom_{\cO}$.

It is more convenient to work with the composition of the usual Koszul duality functor 
with the duality $\dd$ to obtain a contravariant functor 
\begin{equation}\label{KoszuleqO}
\cK_\lambda^\mu:=\mathbf{d}\cK_{A^\mu_\lambda}\,:\,\cD^b({}^{\mZ}\cO^\mu_\lambda)\;\,\tilde\to\;\, \cD^b({}^{\mZ}\cO^\lambda_{\widehat\mu}),
\end{equation}
where we also identify the graded module categories corresponding to the isomorphism 
$E(A^\mu_\lambda)\cong A^{\lambda}_{\widehat\mu}$. This functor satisfies
\begin{equation}\label{KdFshift}
\cK^{\mu}_\lambda(\cM^\bullet[i]\langle j\rangle)=K^{\mu}_\lambda(\cM^\bullet)[j-i]\langle j\rangle, 
\end{equation}
see \cite[Theorem~3.11.1]{BGS}. Similarly, we define 
$\epsilon^\mu_\lambda=\mathbf{d}\epsilon_{_{A^\mu_\lambda}}$, as a contravariant 
equivalence of categories
\begin{equation}\label{eqeps}
\epsilon^\mu_\lambda:\mathfrak{LP}^\mu_\lambda\;\,\tilde\to\;\, {}^{\mZ}\cO^\lambda_{\widehat\mu}.
\end{equation}

We conclude this subsection with the introduction of the graded lifts to ${}^{\mZ}\cO$ of the 
translation functors on~$\cO$, as studied in~\cite{Stroppel}. 
We denote them by the same symbols as on~$\cO$ and use the grading convention of~\cite{Stroppel}. 
This means that the graded version of equation~\eqref{thetaon} is
\begin{equation}\label{normon}
\theta_\lambda^{on} L(x)\langle 0\rangle=L(x\cdot\lambda)\langle -l(w_0^\lambda)\rangle,
\end{equation}
for all $x\in X_\lambda$, and that
\begin{equation}\label{adjtransg}\hom_{\cO_0}(\theta_\lambda^{out}M,N)\cong \hom_{\cO_\lambda}(M,\theta_\lambda^{on}N\langle l(w_0^\lambda)\rangle),\end{equation}
see also \cite[Lemma~38]{MOS}. This implies that
\begin{equation}\label{projout}
\theta_\lambda^{out} P(x\cdot\lambda)\langle 0\rangle=P(x)\langle 0\rangle.
\end{equation}

\subsection{Twisting and shuffling functors}

We will use the twisting functor~$T_s$, which is an endofunctor on each integral block 
$\cO_\lambda$ corresponding to a simple reflection $s$, see~\cite{AS, simple}. 
For any $w\in W$ with reduced expression $w=s_{1}s_2\cdots s_m$, we can 
define the functor
\begin{displaymath}
T_{w}=T_{s_1}T_{s_2}\cdots T_{s_m}, 
\end{displaymath}
where the resulting functor does not depend on the choice of a reduced expression, 
see \cite[Corollary~11]{Khomenko}. The functor $T_w$ is right exact and its 
derived functor $\cL T_w$ is an auto-equivalence of~$\cD^b(\cO_0)$, see 
\cite[Corollary~4.2]{AS}. This property extends to a singular block of 
category~$\cO$, see e.g. \cite[Proposition~5.11]{CM1}, so we have an auto-equivalence
\begin{equation}\label{dertwist}
\cL T_{w}:\, \, \cD^b(\cO_\lambda)\;\tilde\to \;\cD^b(\cO_\lambda).
\end{equation}
Twisting functors admit graded lifts, see \cite[Appendix]{pairing} or \cite[Theorem~1.1]{Khomenko}.
By \cite[Theorem~3.2]{AS}, the following diagram commutes:
\begin{equation}\label{commtwtr}
\xymatrix{
\cD^b(\cO_0)\ar[rr]^{\cL T_w}&&\cD^b(\cO_0)\\
\cD^b(\cO_{\lambda})\ar[rr]^{\cL T_w}\ar[u]^{\theta_\lambda^{out}}&&
\cD^b(\cO_{\lambda})\ar[u]^{\theta_\lambda^{out}}. \\
}
\end{equation}

The shuffling functor $C_s$ corresponding to a simple reflection $s$, 
see~\cite{Carlin, simple}, is the endofunctor of $\cO_0$ defined as the cokernel of the adjunction morphism 
from the identity functor to the projective functor $\theta_s$. 
For any $w\in W$ with reduced expression $w=s_{1}s_2\cdots s_m$, we can 
define the functor
\begin{displaymath}
C_{w}=C_{s_m}C_{s_{m-1}}\cdots C_{s_1}, 
\end{displaymath}
where the resulting functor does not depend on the choice of a reduced expression, 
see \cite[Lemma~5.10]{shuffling} or \cite[Theorem~2]{Khomenko} and 
\cite[Section~6.5]{MOS}. The functor $C_w$ is right exact and $\cL C_w$ is an auto-equivalence, with inverse $\dd\cL C_{w}\dd$, of~$\cD^b(\cO_0)$, see 
\cite[Theorem~5.7]{shuffling}, so
\begin{equation}\label{dershuff0}
\cL C_{w}:\, \, \cD^b(\cO_0)\;\tilde\to \;\cD^b(\cO_0).
\end{equation}

The two basic types of derived auto-equivalences in this sections will be exploited to obtain more complicated derived equivalences in the remainder of the paper. We note that they also appear in more abstract generality in \cite[Theorem~6.15 and Proposition~9.18]{Webster}.

\subsection{Ringel duality for category~$\cO$}\label{prelsecRd}
  
For a quasi-hereditary algebra~$B$, we denote its {\em Ringel dual} as in \cite{Ringel}, by 
\begin{displaymath}
R(B):=\End_{B}(T)^{{\rm opp}},
\end{displaymath} with $T$ the characteristic q.h. tilting module in $B$-mod. 
The Ringel dual is again  
quasi-hereditary. If $B$ is basic, we have $R(R(B))\cong B$, 
see \cite[Theorem~7]{Ringel}.

We also consider the following covariant right exact functor
\begin{displaymath}
\cR_B=\Hom_{B}(\cdot,T)^\ast:\;\,\,B\mbox{-mod}\to R(B)\mbox{-mod},
\end{displaymath}
with $\ast$ being the canonical duality functor from mod-$R(B)$ to $R(B)$-mod. 
We call $\cR_B$ the Ringel duality functor. From \cite[Theorem~6]{Ringel} 
and \cite[Proposition~2.2]{prinjective}, it follows that $\cR_B$ 
restricts to an equivalence between the additive subcategories of 
projective modules of~$B$ and tilting modules of~$R(B)$ and between 
the additive subcategories of tilting modules of~$B$ and injective 
modules of~$R(B)$. It also restricts to an equivalence between the 
category of~$B$-modules with standard flag and the category of~$R(B)$-modules 
with costandard flag. Finally, its left derived functor induces an equivalence 
\begin{displaymath}
\cL \cR_B\,:\; \cD^b(B\mbox{-mod})\;\,\tilde\to\;\,\cD^b(R(B)\mbox{-mod}).
\end{displaymath}

As $\cO_\lambda$ and $\cO^\mu_0$ are Ringel self-dual, it is natural to compose
the Ringel duality functor with a fixed equivalence which realises the self-duality. 
In case $\lambda=0$, we can choose the Ringel duality functor as 
\begin{equation}\label{eqInRi1}
\overline\cR^\mu:= \cL_{l(w_0^\mu)} C_{w_0}: \;\cO_0^\mu \,\to\, \cO_0^\mu,
\end{equation}
see \cite[Proposition~4.4]{prinjective}. In particular, the restriction 
of~$\cL_{l(w_0^\mu)} C_{w_0}$ to the category~$\cO_0^\mu$ is a right exact functor.

In case $\mu=0$, the Ringel duality functor can be interpreted as
\begin{equation}\label{eqInRi2}
\underline\cR_\lambda:= T_{w_0}:\cO_\lambda\to\cO_\lambda,
\end{equation}
see \cite[Section~4.1]{prinjective}.

\subsection{The centre and coinvariants}\label{sscentre}

For an integral dominant $\lambda$, consider $B_\lambda$ and the 
corresponding semisimple Lie algebra $\fg_\lambda$, generated by the 
root spaces of~$\fg$ corresponding to elements in $\pm B_\lambda$. The 
Weyl group of this algebra is isomorphic to $W_\lambda$. Then 
$\fh_\lambda:=\fg_\lambda\cap \fh$ is a Cartan subalgebra in~$\fg_\lambda$. 
Now we can  consider the algebra of coinvariants for~$W_\lambda$ given by
\begin{displaymath}
\mathtt{C}(W_\lambda):= S(\fh_\lambda)/\langle S(\fh_\lambda)_+^{W_\lambda}\rangle. 
\end{displaymath}
The algebra $\mathtt{C}(W_\lambda)$ inherits a positive grading from 
$S(\fh_\lambda)$ which is defined by giving constants degree $0$ and 
elements of~$\fh$ degree $2$.

For the particular case $W_\lambda=W$, we set $\mathtt{C}:=\mathtt{C}(W)$. 
For $w\in W$, let ${}^{w}\mathtt{C}$ denote
the $\mathtt{C}\text{-}\mathtt{C}$ bimodule obtained from
${}_\mathtt{C}\mathtt{C}_\mathtt{C}$ by twisting 
the left action of~$\mathtt{C}$ by $w$ (note that $W$ acts on $\mathtt{C}$ 
by automorphisms). For $X$ a subgroup of $W$,
let $\mathtt{C}^{X}$ denote the algebra of~$X$-invariants in~$\mathtt{C}$. 
Similarly, for a simple reflection $s$, we denote by $\mathtt{C}^{s}$ 
the algebra of $s$-invariants in~$\mathtt{C}$. 

Consider again an integral dominant $\lambda$ and set 
$\mathtt{C}_\lambda:=\mathtt{C}^{W_\lambda}$. By 
\cite[Endomorphismensatz~7]{SoergelD}, we have
$\End_{\cO_\lambda}(P(w_0\cdot\lambda))\cong\mathtt{C}_\lambda$.
We also recall Soergel's combinatorial functor 
\begin{displaymath}
\mathbb{V}_\lambda=\Hom_{\cO}(P(w_0\cdot\lambda),-):\,\cO_0\to 
\mathtt{C}_\lambda\text{-}\mathrm{mod} 
\end{displaymath}
from \cite[Section~2.3]{SoergelD}. By \cite[Theorem~10]{SoergelD}, we have 
\begin{equation}\label{Soergeltrans}
\mathbb{V} \theta_\lambda^{out}\cong \mathtt{C}
\otimes_{\mathtt{C}_\lambda}\mathbb{V}_\lambda\quad\mbox{and}
\quad\mathbb{V}_\lambda\theta_\lambda^{on}\cong 
\res^{\mathtt{C}}_{\mathtt{C}_\lambda}\mV,
\end{equation}
for any integral dominant $\lambda$. All these statement admit canonical graded lifts, 
as the grading on $\cO$ can be introduced via $\mV_\lambda$, see e.g. \cite{Stroppel}.

From $\mathtt{C}_\lambda\cong \End_{\cO_\lambda}(P(w_0\cdot\lambda))$ and \cite[Lemma~6.2]{Brundan} or 
\cite[Theorem~5.2(2)]{prinjective}, it follows that $\mathtt{C}_\lambda$ is isomorphic to the centre of $A_\lambda$.
More generally, we denote the centre of $A^\mu_\lambda$ by $\mathtt{C}^\mu_\lambda$.

In the remainder of this subsection we consider $\mathfrak{g}=\mathfrak{sl}(n)$. 
Then $W\cong S_n$ and $\mathtt{C}^\mu_\lambda$ has been calculated
in~\cite{Brundan, StroppelCentre}. For an integral dominant~$\lambda$, it follows 
from \cite[Theorem~11]{SoergelD} that $\cO_\lambda$ is uniquely determined, 
up to equivalence, by a composition $p(\lambda)=(p_1,\cdots, p_k)$ 
of $n$. This composition is defined by demanding that $W_\lambda$, as a subgroup 
of $S_n$, is naturally given by $S_{p_1}\times S_{p_2}\times \cdots\times S_{p_k}$. It is well-known, by a result of Borel, that $\mathtt{C}_\lambda$ is, as a graded 
algebra, isomorphic to the cohomology ring of a partial flag variety. 
The Hilbert-Poincar\'e polynomial for $\mathtt{C}_\lambda$, with $p(\lambda)$ as above, 
is hence well-known, see e.g. \cite{Chen}, and is given by
\begin{equation}
\label{HilPoin}
\sum_{i=0}^\infty \dim \left(\mathtt{C}_\lambda\right)_{2i} z^i\,=\, 
\left(\prod_{i=1}^n(1-z^i)\right)/\left( \prod_{j=1}^{k}\prod_{l=1}^{p_j}(1-z^l)\right).
\end{equation}




\subsection{Extension quivers in parabolic category~$\cO$}\label{secext1}

We demonstrate the basic property that the $\Ext^1$-quiver of a parabolic singular block can be read off immediately from the $\Ext^{1}$-quiver of the principal block $\cO_0$.

\begin{proposition}
\label{propquiv}
For~$x,y\in X_\lambda^\mu$, we have an isomorphism
\begin{displaymath}
\Ext^1_{\cO_\lambda^\mu}(L(x\cdot\lambda),L(y\cdot\lambda))\cong\Ext^1_{\cO_0}(L(x),L(y)). 
\end{displaymath}
\end{proposition}

\begin{proof}
As $\cO^\mu_\lambda$ is a Serre subcategory of~$\cO_\lambda$, we immediately have
\begin{equation*}
\Ext^1_{\cO_\lambda^\mu}(L(x\cdot\lambda),L(y\cdot\lambda))
\cong\Ext^1_{\cO_\lambda}(L(x\cdot\lambda),L(y\cdot\lambda)).
\end{equation*}
Now, the Koszul duality of~\cite{BGS} implies an isomorphism
\begin{displaymath}
\Ext^1_{\cO_\lambda}(L(x\cdot\lambda),L(y\cdot\lambda))
\cong\Ext^1_{\cO^\lambda_0}(L(x^{-1}w_0),L(y^{-1}w_0)).
\end{displaymath}
Applying the same procedure again, but now to the right-hand side above, concludes the proof.
\end{proof}

The elements of~$\Ext^1_{\cO^\mu_\lambda}(L(x\cdot\lambda),L(y\cdot\lambda))$ correspond 
to morphisms from $P^\mu(y\cdot\lambda)$ to $P^\mu(x\cdot\lambda)$ such that the top 
of~$P^\mu(y\cdot\lambda)$ maps to the top of the radical of~$P^\mu(x\cdot\lambda)$. 
As the (Koszul) grading on $A^\mu_\lambda$ is such that the degree $0$ part is 
semisimple and the algebra is generated by the degree~$0$ and~$1$ parts, this can 
also be expressed as
\begin{displaymath}
\Ext^1_{\cO^\mu_\lambda}(L(x\cdot\lambda),L(y\cdot\lambda))
\cong\hom_{\cO^\mu_\lambda}(P^\mu(y\cdot\lambda)\langle1\rangle,P^\mu(x\cdot\lambda)).
\end{displaymath}
Proposition~\ref{propquiv} and equation~\eqref{projout} then imply the following corollary.

\begin{corollary}\label{deg1mor}
For~$x,y\in X_\lambda^\mu$ with $P:=P^\mu(x\cdot\lambda)$ and 
$Q:=P^\mu(y\cdot\lambda)$, the graded translation functor $\theta_\lambda^{out}$ 
induces an isomorphism
\begin{displaymath}
\hom_{\cO^\mu_\lambda}(P\langle 1\rangle,Q  )\cong \hom_{\cO^\mu_0}
(\theta_\lambda^{out}P\langle 1\rangle,\theta_\lambda^{out}Q). 
\end{displaymath}
\end{corollary}


\section{Graded versus non-graded derived equivalences}\label{secgranongra}

\subsection{The example of the Koszul duality functor}

The Koszul duality functor in equation~\eqref{Koszuleq} gives an equivalence between 
bounded derived categories of graded modules. We demonstrate, however, that Koszul dual 
algebras are, in general, not derived equivalent as ungraded algebras. 

\begin{proposition}\label{Kdnongrad}
Despite the existence of an equivalence of triangulated categories 
$\cD^b({}^{\mZ}\cO_\lambda^\mu)\cong \cD^b({}^{\mZ}\cO_{\widehat\mu}^\lambda)$ 
as given in  \eqref{KoszuleqO}, in general, the categories $\cD^b(\cO_\lambda^\mu)$ 
and $\cD^b(\cO_{\widehat\mu}^\lambda)$ are not equivalent as triangulated categories.
\end{proposition}

\begin{proof}
Note that~$A_\lambda^\mu$ and $A_{\widehat\mu}^\lambda\cong E(A^\mu_\lambda)$ 
are not derived equivalent, if their centres are not isomorphic, see \cite[Proposition~9.2]{Rickard}.
Consider the socle of the left regular module for the two centres. For $A_\lambda$
this socle is simple, by Subsection~\ref{sscentre}. It follows easily from \cite[Lemma~6.2]{Brundan}
that the number of simple modules in that socle for $A^\lambda$ must be at least
the number of non-isomorphic indecomposable projective-injective modules.
This is the cardinality of the right Kazhdan-Lusztig 
cell of $w_0^\lambda w_0$.
\end{proof}

\subsection{Gradable derived equivalences}

In this subsection we define a special case of the concept of a derived equivalence 
(see \cite[Definition~6.5]{Rickard}), which we will investigate for category 
$\cO$ further in the paper.

\begin{definition}\label{defgradeq}
Two finite dimensional graded algebras $B$ and $D$ are said to be gradable 
derived equivalent if one of the following equivalent properties is satisfied.
\begin{enumerate}[$($i$)$]
\item\label{refnew21} There is a triangulated  equivalence from 
$\cD^b(B\mbox{-\rm gmod})$ to $\cD^b(D\mbox{-\rm gmod})$ which commutes with 
the degree shift functor $\langle 1\rangle$.
\item\label{refnew21-2} There is a triangulated  equivalence from 
$\cD^b(B\mbox{-\rm mod})$ to $\cD^b(D\mbox{-\rm mod})$, with inverse $G$, 
such that $F$ and $G$ admit graded lifts.
\end{enumerate}
\end{definition}

\begin{remark}
{\rm 
It is clear that the failure of the Koszul duality functor to satisfy 
the requirement in Definition~\ref{defgradeq}\eqref{refnew21}, see equation \eqref{KdFshift}, is precisely 
what prevents it from being the graded lift of an ungraded equivalence.
}
\end{remark}

Now we state and prove two simple propositions which demonstrate 
equivalence of the two properties in Definition~\ref{defgradeq}.

\begin{proposition}
\label{descequiv}
Consider two finite dimensional graded algebras $B$ and~$D$ such that there is an equivalence
\begin{displaymath}
\widetilde F:\cD^b(B\mbox{-\rm gmod})\to \cD^b(D\mbox{-\rm gmod}).
\end{displaymath}
If $\widetilde F$ commutes with  $\langle 1\rangle$, 
then $\tilde F$ is a graded lift of an equivalence
\begin{displaymath}
F:\cD^b(B\mbox{-\rm mod})\to \cD^b(D\mbox{-\rm mod}).
\end{displaymath}
Moreover, if $\widetilde G$ is inverse to $\widetilde F$, then $\widetilde G$ is the graded lift of an inverse $G$ to $F$.
\end{proposition}

\begin{proof}
Take a minimal projective generator $P_D$ for $D$-mod and the graded lift $\widetilde P_D$.
Denote by $\cT^\bullet$ the object in $\cD^b(B\mbox{-mod})$ obtained by forgetting the grading on
$\widetilde{F}^{-1}(P_D^\bullet)$. One deduces straightforwardly that
$\cT^\bullet$ is a tilting complex according to \cite[Definition~6.5]{Rickard} with $D^{\mathrm{opp}}\cong \End_{\cD^b(B\mbox{-}{\rm mod})}(\cT^\bullet)$.
The first assertion of the proposition then follows from \cite[Theorem~6.4]{Rickard}.

Now we consider the inverse $\widetilde G$. From $\widetilde F\widetilde G\cong \Id\cong\widetilde G\widetilde F$ and $\widetilde F\langle1\rangle\cong \langle 1\rangle \widetilde F$ it follows that 
$\widetilde G\langle 1\rangle\cong\langle 1\rangle\widetilde G.$
Hence we can use the first part to obtain that $\widetilde G$ is the graded lift of a triangulated functor $G$. The construction, moreover, implies that $FG$ and $GF$ acts as identity functors restricted to the subcategories of projective modules. As they are triangulated functors it follows that $F$ and $G$ are mutually inverse.
\end{proof}

\begin{proposition}\label{liftequiv}
Let $B$ and $D$ be finite dimensional graded algebras and 
\begin{displaymath}
F:\;\cD^b(B\mbox{-\rm mod})\;\to\; \cD^b(D\mbox{-\rm mod})
\end{displaymath}
an equivalence with inverse $G$. Assume that $F$ and $G$ 
admit graded lifts $\widetilde F$ and $\widetilde G$, respectively.
Then $\widetilde F$ gives an equivalence of triangulated categories
\begin{displaymath} 
\widetilde F:\cD^b(B\mbox{-\rm gmod})\to \cD^b(D\mbox{-\rm gmod}).
\end{displaymath}
\end{proposition}

\begin{proof}
First we prove that $\widetilde F$ is essentially surjective (dense). 
Consider the forgetful functor $f_D:\cD^b(D\mbox{-gmod})\to \cD^b(D\mbox{-mod})$ and an indecomposable object $\cX^\bullet$ in $\cD^b(D\mbox{-gmod})$ 
and set $\cY^\bullet:=\widetilde G\cX^\bullet$. As we have
$f_D \widetilde F\widetilde G\cong f_D$, it follows that $f_D\widetilde F (\cY^\bullet)\;\cong \;f_D\cX^\bullet.$
By \cite[Lemma~2.5.3]{BGS}, the proof of which extends easily to the 
derived category, we then find that there is $j\in\mN$ such that 
$\cX^\bullet \cong \widetilde F(\cY^\bullet)\langle j\rangle$, or
\begin{displaymath}
\widetilde F (\cY^\bullet \langle j\rangle)\;\cong \cX^\bullet.
\end{displaymath}
The density of $\widetilde F$ hence follows.

For $\cX^\bullet, \cY^\bullet\in \cD^b(B\mbox{-gmod})$, we have the commutative diagram
\begin{displaymath}
\xymatrix{
\bigoplus_{j\in\mN}\Hom_{\cD^b(B\mbox{-}\rm{gmod})}(\cX^\bullet,\cY^\bullet
\langle j\rangle)\ar[r]^{\widetilde F}\ar[d]^{ f_B}&  
\bigoplus_{j\in\mN}\Hom_{\cD^b(D\mbox{-}\rm{gmod})}
(\widetilde{F}\cX^\bullet,\widetilde{F}\cY^\bullet\langle j\rangle)\ar[d]^{f_D}\\
\Hom_{\cD^b(B\mbox{-}\rm{mod})}(f_B\cX^\bullet,f_B\cY^\bullet)\ar[r]^{F}&
\Hom_{\cD^b(D\mbox{-}\rm{mod})}(f_D\widetilde{F}\cX^\bullet,f_D\widetilde{F}\cY^\bullet )\\
} 
\end{displaymath}
Here $F$, $f_B$ and $f_D$ act by isomorphisms and hence so does $\widetilde F$. As $\widetilde F$ respects it then follows easily that $\widetilde F$ is full and faithful.
Hence $\widetilde F$ is an equivalence of categories. By construction it 
is also an equivalence of triangulated categories.
\end{proof}

Finally, we note the following immediate extension of \cite[Proposition~9.2]{Rickard}.

\begin{lemma}\label{gradedcentre}
Consider two finite dimensional graded algebras $B,D$ which are gradable derived equivalent. Then there is an isomorphism of graded algebras $\cZ(B)\cong \cZ(D)$, with canonically inherited grading on the centres.
\end{lemma}

\subsection{An application: derived shuffling in the parabolic setting}

We use the results of the previous subsection to extend  \eqref{dershuff0} to the form of \eqref{dertwist} 
and \eqref{commtwtr}. First we point out two subtleties (a) and (b). Consider the exact inclusion $\imath^\mu:\cO^\mu_0\to\cO_0$ leading to a faithful (see e.g. \cite[Lemma 2.6]{Back}) triangulated functor 
\begin{equation*}\label{-imatheq}
\imath^\mu:\;\cD^b(\cO_0^\mu)\;\to\;\cD^b(\cO_0).
\end{equation*}
Hence $\cD^b(\cO_0^\mu)$ is canonically equivalent to a subcategory 
of $\cD^b(\cO_0)$.
 
(a) By construction, $\cO^\mu_0$ is a full Serre subcategory of $\cO_0$ and $C_s$ restricts to a right exact endofunctor of 
$\cO^\mu_0$, which we denote by $C^{\mu}_s$. However, its left derived 
functor $\cL C^{\mu}_s$ need not be isomorphic to the restriction of 
$\cL C_s$ to $\cD^b(\cO_0^\mu)$, viewed as a subcategory 
as above. A trivial example is $\fg=\mathfrak{sl}(2)$, as then 
$C_s^{\mu}\cong 0$ for $\mu$ singular. However, a restriction of $\cL C_s$ 
is never zero as $\cL C_s$ is an equivalence.

(b) The objects of the subcategory $\cD^b(\cO_0^\mu)$ are 
the complexes in $\cD^b(\cO_0)$ for which the module in each 
position is a module in the subcategory $\cO_0^\mu$ of $\cO_0$. 
However, this subcategory is neither {\em full}, nor {\em isomorphism closed}. 
It is not full as, for instance, there 
can be higher extensions in $\cO_0$ between projective objects in 
$\cO_0^\mu$. It is not isomorphism closed, see, for instance, the projective resolution in $\cD^b(\cO_0)$ of a 
module in $\cO^\mu_0$. The functor
$\cL C_{s}$ maps, by the definition of a derived functor, objects in 
$\cD^b(\cO_0^\mu)$ to something only {\em isomorphic} to objects in 
$\cD^b(\cO_0^\mu)$. To properly define a restriction of $\cL C_{s}$ to $\cD^b(\cO_0^\mu)$ 
is hence a non-trivial problem.

\begin{proposition}\label{commdiagshuff}
For any $w\in W$, there is an endofunctor $\cL C_{w}$ of $\cD^b(\cO_0^\mu)$, which yields an auto-equivalence and admits a commuting diagram
\begin{equation*}
\xymatrix{
\cD^b(\cO_0)\ar[rr]^{\cL C_w}&&\cD^b(\cO_0)\\
\cD^b(\cO_{0}^\mu)\ar[rr]^{\cL C_w}\ar[u]^{\imath^\mu}&&
\cD^b(\cO_{0}^\mu)\ar[u]^{\imath^\mu}. \\
}
\end{equation*}
The same holds in the graded setting.
\end{proposition}

\begin{proof}
An inverse of $\cL T_w$ is given by by $\dd\cL T_{w}\dd$, 
see \cite[Section~4]{AS}. Proposition~\ref{liftequiv} 
thus implies that the graded lift of $\cL T_w$ induces an 
auto-equivalence of $\cD^b({}^{\mZ}\cO_\lambda)$. So the diagram \eqref{commtwtr} admits a graded lift with equivalences
 on the horizontal arrows.
Then we apply \cite[Sections 6.4 and 6.5]{MOS}, see also \cite{RH} 
or Proposition~\ref{RHresult}. This implies a commutative diagram, 
where the horizontal arrows are equivalences
\begin{displaymath}
\xymatrix{
\cD^b({}^{\mZ}\cO_0)\ar[rr]^{\cL C_w}&&\cD^b({}^{\mZ}\cO_0)\\
\cD^b({}^{\mZ}\cO_{0}^\mu)\ar[rr]^{\widetilde F}\ar[u]^{\imath^\mu}&&
\cD^b({}^{\mZ}\cO_{0}^\mu)\ar[u]^{\imath^\mu}, \\
}
\end{displaymath}
for $\widetilde{F}:=(\cK^{\widehat\mu})^{-1}\circ\cL T_{w}\circ\cK^{\widehat\mu}$. 
As $\cL T_w$ is a triangulated functor which commutes with $\langle 1\rangle$, equation \eqref{KdFshift} implies that $\widetilde{F}$ commutes with $\langle 1\rangle$, so 
Proposition~\ref{descequiv} implies that $\widetilde{F}$ is the lift of an ungraded 
equivalence and hence a gradable derived equivalence.
\end{proof}
 
\begin{remark}{\rm
Consider the complex $0\to \mathrm{Id}_{\cO_0}\to \theta_s\to 0$ of exact functors, 
where the non-zero map is given by the adjunction morphism, {\it cf.} \cite{Rickard2} and
\cite[Remark~5.8]{shuffling}. Applying this to a complex in 
$\cD^b(\cO_0)$ and taking the total complex defines a triangulated endofunctor of 
$\cD^b(\cO_0)$, isomorphic to $\cL C_{s}$. This endofunctor, by construction,
preserves the image of $\cD^b(\cO_0^\mu)$ under $\imath^\mu$, since both the identity and
$\theta_s$ do. This gives an alternative construction of the equivalence in 
Proposition~\ref{commdiagshuff}.
}
\end{remark}

\section{Graded translation functors}\label{secGradTrans}

\subsection{Translation through the principal block}

\begin{proposition}\label{thetatheta}
As graded functors, we have
\begin{displaymath}
\theta^{on}_\lambda\theta^{out}_\lambda\langle 0\rangle\,\cong\, 
\bigoplus_{j\in\mN}\, \Id_{\cO_\lambda}^{\oplus c_j}\langle 
j-l(w_0^\lambda)\rangle,\quad\mbox{with}\quad c_j:= 
\dim \left(\mathtt{C}(W_\lambda)\right)_{j}.
\end{displaymath}
In particular, $\pm l(w_0^\lambda)$ are exactly the extremal degrees in 
which $\Id$ appears.
\end{proposition}

\begin{proof}
The ungraded statement follows easily from \cite[Theorem~3.3]{BG} and \cite[Formula~4.13(1)]{Jantzen}.
It thus suffices to prove that
\begin{displaymath}
\mV_\lambda \theta_\lambda^{on}\theta_{\lambda}^{out}\;\cong\;\bigoplus_{j\in\mN}\, 
\mV_{\lambda}^{\oplus d_j}\langle 
j-l(w_0^\lambda)\rangle,\quad\mbox{for some}\quad d_j\ge 
\dim \left(\mathtt{C}(W_\lambda)\right)_{j}.
\end{displaymath}
Using equation \eqref{Soergeltrans} and ignoring an overall grading shift, we find
\begin{displaymath}
\mV_\lambda \theta_\lambda^{on}\theta_{\lambda}^{out}\;\cong\; 
\mathtt{C}\otimes_{\mathtt{C}_\lambda}\; \mV_\lambda.
\end{displaymath}
So it suffices to prove that 
\begin{displaymath}
\dim \mathtt{C}_{2i}-\dim(\mathtt{C}_\lambda)_{2i}+\dim (\mC)_{2i}\ge\dim \mathtt{C}(W_\lambda)_{2i}, 
\end{displaymath}
for all $i\in\mN$, where, of course, $\dim (\mC)_{2i}=\delta_{i0}$. Consider a $W_\lambda$-equivariant 
morphism $\fh\tto\fh_\lambda$. This extends to a graded morphism $S(\fh)\tto S(\fh_\lambda)$ which 
we compose with the canonical surjection $S(\fh_\lambda)\tto \mathtt{C}(W_\lambda)$ to find
$\xi: S(\fh)\tto \mathtt{C}(W_\lambda).$ By construction $S(\fh)_+^{W}$, and hence the 
corresponding ideal, is in the kernel of $\xi$. This implies a morphism of graded algebras
\begin{displaymath}
\eta: \mathtt{C}\tto \mathtt{C}(W_\lambda). 
\end{displaymath}
As $(\mathtt{C}_\lambda)_+=\mathtt{C}^{W_\lambda}_+$ is in the kernel of $\eta$, 
this proves the desired inequalities.
\end{proof}

As an application, we prove the following proposition, which we need later.

\begin{proposition}\label{deg01}
Consider two objects $\cX^\bullet,\cY^\bullet$ of~$\cD^b({}^\mZ\cO_\lambda^\mu)$ such that 
\begin{displaymath}
\hom_{\cD^b(\cO_\lambda^\mu)}(\cX^\bullet, \cY^\bullet\langle j\rangle)=0\qquad\mbox{if}\quad j>0. 
\end{displaymath}
Then $\theta_\lambda^{out}$ induces isomorphisms
\begin{eqnarray*}
\hom_{\cD^b(\cO_0^\mu)}(\theta_\lambda^{out}\cX^\bullet,\theta_\lambda^{out}\cY^\bullet)&\cong &\hom_{\cD^b(\cO_\lambda^\mu)}(\cX^\bullet,\cY^\bullet),\\
\hom_{\cD^b(\cO_0^\mu)}(\theta_\lambda^{out}\cX^\bullet\langle 1\rangle,\theta_\lambda^{out}\cY^\bullet)&\cong &\hom_{\cD^b(\cO_\lambda^\mu)}(\cX^\bullet\langle 1\rangle,\cY^\bullet).
\end{eqnarray*}
\end{proposition}

\begin{proof}
Consider $i\in\mZ$, then equation~\eqref{adjtransg} and Proposition~\ref{thetatheta} 
imply that
\begin{displaymath}
\hom_{\cD^b(\cO_0^\mu)}(\theta_\lambda^{out}\cX^\bullet\langle i\rangle,\theta_\lambda^{out}\cY^\bullet)
\cong \hom_{\cD^b(\cO_\lambda^\mu)}(\cX^\bullet\langle i\rangle,\bigoplus_{j\in\mN}
(\cY^\bullet)^{\oplus c_j}\langle j\rangle).
\end{displaymath}
Now, if $i=0$, the result follows by the assumptions as $c_0=1$. 
If $i=1$, the result follows from the assumptions and $c_1=0$ and $c_0=1$.
\end{proof}

This proposition generalises Corollary~\ref{deg1mor} and hence provides an alternative proof.

\subsection{Translating standard modules}

In this subsection we completely describe the graded translation of standard modules 
to and from the wall. These results generalise 
\cite[Theorem~8.1(3) and Theorem~8.2(2)]{Stroppel} to arbitrary walls and 
the parabolic setting. Consider the bijection
\begin{displaymath}
b_\lambda:\, W\to X_\lambda\times W_\lambda, 
\end{displaymath}
which is inverse to multiplication and denote by $b_\lambda^1$ and $b_\lambda^2$
the composition of~$b_\lambda$ with the projection on the 
$X_\lambda$-component and the $W_\lambda$-component, respectively.

\begin{theorem}\label{thmDeltaon}
For any $x\in X^\mu$, we have
\begin{displaymath}
\theta^{on}_\lambda \Delta^{\mu}(x)\cong\begin{cases}\Delta^{\mu}(b^1_\lambda(x)\cdot\lambda)
\langle l(b^2_\lambda(x))- l(w_0^\lambda)\rangle, & \mbox{ if } b^1_\lambda(x)\in X^\mu;\\0,&\mbox{otherwise.} \end{cases}
\end{displaymath}
In particular, for any $y\in X_\lambda$ and $u\in W_\lambda$, we have
\begin{displaymath}
\theta^{on}_\lambda \Delta(yu)\,\cong \, \Delta(y\cdot\lambda)\langle l(u) -l(w_0^\lambda)\rangle. 
\end{displaymath}
\end{theorem}

\begin{theorem}\label{thmDeltaout}
For any $x\in X^\mu_\lambda$, the standard filtration 
of~$\theta_\lambda^{out}\Delta^\mu(x\cdot\lambda)$ satisfies
\begin{displaymath}
\left(\theta_{\lambda}^{out}\Delta^{\mu}(x\cdot\lambda):
\Delta^\mu(xu)\langle j\rangle\right)=\delta_{j,l(u)}\qquad \text{ for all }\quad u\in W_\lambda, 
\end{displaymath}
moreover, there are no other standard modules appearing in the filtration.
\end{theorem}

In the remainder of this subsection we prove these two theorems.

\begin{lemma}\label{on2}
For any $x\in X_\lambda$, $u\in W_\lambda$ and $j\in\mN$, we have
\begin{displaymath}
[\Delta(xu):L(x)\langle j\rangle]=\delta_{j,l(u)}. 
\end{displaymath}
\end{lemma}

\begin{proof}
We prove the Koszul dual statement, which is 
\begin{displaymath}
\dim\Ext^j_{\cO}(\Delta(vy),L(y))=\delta_{j,l(v)}, 
\end{displaymath}
for any $y\in X^\lambda$, $v\in W_\lambda$ and $i\in\mN$. 
That the left-hand side is zero, for~$j>l(v)$, follows 
immediately from  \cite[Theorem~6.11]{Humphreys}. 
For any simple reflection~$s\in W_\lambda$ and 
$v\in W_\lambda$ such that $sv>v$, the procedure 
in the proof of \cite[Proposition~3]{MR2366357} shows
\begin{displaymath}
\dim \Ext^k_{\cO}(\Delta(svy),L(y))\;=\;\dim\Ext^{k-1}_{\cO}(\Delta(vy),L(y)),
\end{displaymath}
which proves the claim inductively.
\end{proof}

\begin{proof}[Proof of Theorem~\ref{thmDeltaon}]
It suffices to prove the non-parabolic case, as the remainder follows 
via an application of the Zuckerman functor, which commutes with translation functors.
The non-parabolic result follows from equation~\eqref{normon}, 
Lemma~\ref{on2} and \cite[Formula~4.12(2)]{Jantzen}.
\end{proof}

\begin{proof}[Proof of Theorem~\ref{thmDeltaout}]
For any $y\in X^\mu$, \cite[Theorem~3.3(d)]{Humphreys} and 
equation~\eqref{adjtransg} yield
\begin{displaymath}
\left(\theta_{\lambda}^{out}\Delta^{\mu}(x\cdot\lambda):\Delta^\mu(y)\langle j\rangle\right)=
\hom_{\cO^\mu_\lambda}(\Delta^{\mu}(x\cdot\lambda),\theta_\lambda^{on}\nabla^\mu(y)\langle j+l(w_0^\lambda)\rangle). 
\end{displaymath}
Theorem~\ref{thmDeltaon} then implies that the only $y$ which can appear 
non-trivially are those for which we have 
$b^1_\lambda(y)=x$, which is precisely the set $x W_\lambda\subset X^\mu$.
\end{proof}

\begin{remark}{\rm
By using Proposition~\ref{RHresult}, an 
alternative proof of Theorem~\ref{thmDeltaout} would be to determine a resolution
of standard modules in~$\cO^\mu$ by Verma modules. This can be obtained by 
applying parabolic induction to~$\fg$ of the BGG resolutions for the Levi subalgebra 
of $\fq_\mu$, see e.g. \cite[Section~6]{Humphreys}.}
\end{remark}


\subsection{Koszul duality}

In this subsection we derive a slight generalisation of \cite[Theorem~4.1]{RH} 
and \cite[Theorem~35]{MOS}.

\begin{proposition}\label{RHresult}
There are commutative diagrams of functors as follows:
\begin{displaymath}
\xymatrix{
\cD^b({}^{\mZ}\cO^\mu_\lambda)\ar[rr]^{\imath^\mu}&&\cD^b({}^{\mZ}\cO_\lambda)\ar@{-}[d]^{K^{\widehat\lambda}}&& 
\cD^b({}^{\mZ}\cO_\lambda)\ar[rr]^{\cL Z^\mu}&&\cD^b({}^{\mZ}\cO^\mu_\lambda)\ar@{-}[d]^{K^{\widehat\lambda}_\mu}\\
\cD^b({}^{\mZ}\cO^{\widehat\lambda}_{\mu})\ar[rr]^{\theta^{out}_{\mu}}\ar[u]^{K^{\widehat\lambda}_{\mu}}&&
\cD^b({}^{\mZ}\cO^{\widehat\lambda}_0)\ar[u]&&      \cD^b({}^{\mZ}\cO^{\widehat\lambda}_0)\ar[rr]^{\theta^{on}_{\mu}\langle l(w_0^\mu)\rangle}\ar[u]^{K^{\widehat\lambda}}&&\cD^b({}^{\mZ}\cO_{\mu}^{\widehat\lambda})\ar[u] \\
}
\end{displaymath}
\end{proposition}

Before proving this we list some consequences. Recall the $c_j$'s from Proposition~\ref{thetatheta}.

\begin{corollary}\label{Zi}
On~$\cD^b({}^{\mZ}\cO_\lambda^\mu)$, we have 
\begin{displaymath}
\cL Z^\mu \circ \imath^\mu\cong\bigoplus_{j\in\mN} \Id_{{}^{\mZ}\cO^\mu_\lambda}^{\oplus c_j}[j]\langle j\rangle. 
\end{displaymath}
For any $M$ in~${}^{\mZ}\cO_\lambda$ which is locally $U(\fq^\mu)$-finite and for any $k\in\mN$, we have
\begin{displaymath}
\cL_kZ^\mu M\cong M^{\oplus c_k}\langle k\rangle. 
\end{displaymath}
\end{corollary}

\begin{proof}
This is a direct application of Proposition~\ref{RHresult} to Proposition~\ref{thetatheta}. 
In particular, the maps in the complex $\cL Z^\mu \circ \imath^\mu(N^\bullet)$ are trivial 
for any $N\in{}^{\mZ}\cO^\mu_\lambda$, which yields the result about the cohomology functors.
\end{proof}

\begin{remark}{\rm
\label{Zinongrad}
Note that Corollary~\ref{Zi} also implies the isomorphism
\begin{displaymath}
\cL Z^\mu \circ \imath^\mu\cong\bigoplus_{j\in\mN} 
\Id_{\cO^\mu_\lambda}^{\oplus c_j}[j]\;\quad\;\mbox{ on }\quad\cD^b(\cO_\lambda^\mu)
\end{displaymath}
as all projective modules and homomorphism spaces between them are gradable.
}
\end{remark}

\begin{corollary}\label{imathend}
Consider two objects $\cX^\bullet$ and $\cY^\bullet$ in $\cD^b(\cO_\lambda^\mu)$ such that 
\begin{displaymath}
\Hom_{\cD^b(\cO_\lambda^\mu)}(\cX^\bullet,\cY^\bullet[k])=0,\qquad \text{ for all }\quad \;k\not=0. 
\end{displaymath}
Then $\imath^\mu$ induces an isomorphism 
\begin{displaymath}
\imath^\mu\,:\;\Hom_{\cD^b(\cO_\lambda^\mu)}
(\cX^\bullet,\cY^\bullet)\;\,\tilde\to\;\,\Hom_{\cD^b(\cO_\lambda)}
(\imath^\mu\cX^\bullet,\imath^\mu\cY^\bullet). 
\end{displaymath}
\end{corollary}

\begin{proof}
As $\imath^\mu$ is faithful, the morphism is always injective. 
So it suffices to prove that the dimensions of the spaces of homomorphisms agree.
This follows from 
\begin{displaymath}
\Hom_{\cD^b(\cO_\lambda)}(\imath^\mu\cX^\bullet,\imath^\mu\cY^\bullet)\;\cong\;\Hom_{\cD^b(\cO_\lambda)}(\cL Z^\mu\imath^\mu\cX^\bullet,\cY^\bullet), 
\end{displaymath}
given by adjunction, and Remark \ref{Zinongrad}.
\end{proof}

\begin{corollary}\label{corisomclosed}
Consider two objects $\cX^\bullet$ and $\cY^\bullet$ in $\cD^b(\cO_\lambda^\mu)$. If either 
\begin{displaymath}
\theta_\lambda^{out}\cX^\bullet\cong\theta_\lambda^{out}
\cY^\bullet\,\,\mbox{ inside }\cD^b(\cO_0^\mu)\quad \text{ or }\quad
\imath^\mu\cX^\bullet\cong \imath^\mu \cY^\bullet\,\,\mbox{ inside }\cD^b(\cO_\lambda),
\end{displaymath}
then $\cX^\bullet$ and $\cY^\bullet$ are isomorphic in $\cD^b(\cO_\lambda^\mu)$. 
The same is true for $\cD^b({}^{\mZ}\cO^\mu_\lambda)$.
\end{corollary}

\begin{proof}
The property in the graded setting follows immediately from 
Proposition~\ref{thetatheta} and Corollary \ref{Zi}. The ungraded 
version now follows from Remark~\ref{Zinongrad}.
\end{proof}

Now we turn to the proof of Proposition~\ref{RHresult}.

\begin{lemma}\label{embLP}
The translation functor 
$\theta_\lambda^{out}:\cD^b({}^{\mZ}\cO_\lambda^\mu)\to \cD^b({}^{\mZ}\cO^\mu_0)$ 
restricts to a full and faithful functor
$\theta_\lambda^{out}:\;\;\mathfrak{LP}^\mu_\lambda\;\to\;\mathfrak{LP}^\mu$, the image
of which is the full subcategory of $\mathfrak{LP}^\mu$ given by 
linear complexes of projective modules in ${\rm add}(\theta_\lambda^{out}P_\lambda^\mu)$.
\end{lemma}

\begin{proof}
Recall that the categories of 
linear complexes of projective modules are full subcategories of the derived 
categories. By equation \eqref{projout}, $\theta_\lambda^{out}$ 
restricts to a functor from $\mathfrak{LP}^\mu_\lambda$ to 
$\mathfrak{LP}^\mu$. To prove that this is full and faithful, 
consider two linear complexes $\cP^\bullet$ and $\cQ^\bullet$ of projective 
modules in~${}^\mZ\cO^\mu_\lambda$. Proposition~\ref{deg01} now implies 
\begin{displaymath}
\theta_\lambda^{out}:\; \hom_{\mathfrak{LP}_\lambda^\mu}
(\cP^\bullet, \cQ^\bullet)\,\, \tilde\to\,\,
\hom_{\mathfrak{LP}^\mu}
(\theta^{out}_\lambda\cP^\bullet,\theta^{out}_\lambda\cQ^\bullet). 
\end{displaymath}
The description of the image is a consequence of Corollary~\ref{deg1mor}.
\end{proof}
\begin{corollary}
\label{corcor}
The following is a commutative diagram of functors:
\begin{displaymath}
\xymatrix{
{}^{\mZ}\cO^\mu_\lambda\ar[rr]^{\imath^\mu}\ar[d]&&{}^{\mZ}
\cO_\lambda\ar[d]^{(\epsilon^{\widehat\lambda})^{-1}}\\
\mathfrak{LP}^{\widehat\lambda}_{\mu}\ar[rr]^{\theta^{out}_{\mu}}
\ar@{-}[u]^{(\epsilon^{\widehat\lambda}_{\mu})^{-1}}&&\mathfrak{LP}^{\widehat\lambda}\ar@{-}[u]   \\
   }
\end{displaymath}\vspace{-2mm}
\end{corollary}

\begin{proof}
Using the standard properties of $\cK_\mu^{\widehat\lambda}$, see \cite[Theorem 3.11.1]{BGS}, and equation \eqref{projout} implies that, for any $x\in X^\mu_{\lambda}$,
\begin{displaymath}
\theta_{\mu}^{out} (\epsilon^{\widehat\lambda}_\mu)^{-1}L(x\cdot\lambda)\;\cong\;  P^{\widehat\lambda}(w_0 x^{-1})^\bullet\;\cong\;(\epsilon^{\widehat\lambda})^{-1}L(x\cdot\lambda).
\end{displaymath}
Corollary \ref{deg1mor} can then be applied to show that the complexes corresponding to $\theta_{\mu}^{out} (\epsilon^{\widehat\lambda}_\mu)^{-1}M$ and $(\epsilon^{\widehat\lambda})^{-1}\imath^\mu M$, for any $M$ in ${}^{\mZ}\cO^\mu_\lambda$, are isomorphic. As, by construction, this isomorphism is natural, this concludes the proof.
\end{proof}

\begin{proof}[Proof of Proposition~\ref{RHresult}]
The first diagram follows from \cite[Theorem~30 and Proposition~21]{MOS}, 
in combination with Corollary~\ref{corcor}. 
The second diagram is the adjoint reformulation of the first.
\end{proof}

\subsection{A generalisation of Bott's theorem}
Take $\rho$ to be the half of the sum of positive roots. 
Then $Z^{-\rho}$ is the Zuckerman functor to the category of 
finite dimensional modules. Bott's extension of the Borel-Weil theorem then reads
\begin{displaymath}
\cL_kZ^{-\rho}(\Delta(x))=\delta_{k,l(x)} L(e)\quad\mbox{and}\quad 
\cL_kZ^{-\rho}(\Delta(y\cdot\lambda))=0,\quad\text{ for all }\quad k\in\mN ,
\end{displaymath}
for all $x\in W$ and all $y\in X_\lambda$ with $\lambda$ singular. 
For an overview of the approach with Zuckerman functors, see e.g. \cite{Co, Santos}.
We can now generalise this. Let
\begin{displaymath}
b^\mu: W\to W_\mu \times X^\mu 
\end{displaymath}
denote the bijection defined as the inverse of multiplication. 
Define $x_1$ and $x^1$ by $b^\mu(x)=(x_1,x^1)$ as in \cite[Proposition 3.4]{Lepowsky}.

\begin{theorem}\label{Bott}
For any $x\in X_\lambda$ and any integral dominant $\mu$, we have
\begin{displaymath}
\cL_k Z^\mu(\Delta(x\cdot\lambda))=
\begin{cases}\delta_{k,l(x_1)}\Delta^\mu(x^1\cdot\lambda),& \mbox{if }x^1\in X_\lambda;\\ 
0,&\mbox{otherwise}.\end{cases}
\end{displaymath}
In particular, for any $y\in X^\mu$ and $u\in W_\mu$, we have
\begin{displaymath}
\cL_k Z^\mu(\Delta(uy))=\delta_{k,l(u)}\Delta^\mu(y).
\end{displaymath}
\end{theorem}

Setting $\mu=-\rho$ (so $x=x_1$ and $x^1=e$) yields the original result of~Bott.

\begin{proof}
This follows from a direct application of the Koszul duality functor to 
the statement in Theorem~\ref{thmDeltaon} by using Proposition~\ref{RHresult}.
\end{proof}

\begin{remark}{\rm
This could also be proved directly by using parabolic induction
to reduce to the Bott's theorem~for the Levi subalgebra. 
This would yield an alternative proof of Theorem~\ref{thmDeltaon} 
by applying Proposition~\ref{RHresult}.}
\end{remark}

\subsection{Translation through the wall}

In this subsection we gather some technical results on translation 
through the wall which will be needed later and which can be 
formulated most generally by ignoring grading.

\begin{lemma}
Consider an indecomposable $\cX^\bullet\in\Ob(\cD^b(\cO_0))$ 
such that we have $\theta_{w_0^\lambda}\cX^\bullet\cong \left(\cX^\bullet\right)^{\oplus |W_\lambda|}$. 
Then there is an indecomposable $\cY^\bullet\in \Ob(\cD^b(\cO_\lambda))$ 
and $k\in\mN$ such that
$\theta_\lambda^{on} \cX^\bullet\cong\left(\cY^\bullet\right)^{\oplus k}$.
\end{lemma}

\begin{proof}
We take $\cY^\bullet$ to be some indecomposable direct summand of 
$\theta^{on}_\lambda\cX^\bullet$. Then there is some 
$\cA^\bullet\in\Ob(\cD^b(\cO_\lambda))$ such that
\begin{displaymath}
\theta^{on}_\lambda\cX^\bullet\;\cong\;\cY^\bullet\oplus\cA^\bullet.
\end{displaymath}
Applying $\theta^{out}_\lambda$ to the above isomorphism implies that 
there must be some $p\in\mN$ for which
$\theta^{out}_\lambda\cY^\bullet\cong\left(\cX^\bullet\right)^{\oplus p}.$
Applying $\theta^{on}_\lambda$, while using the ungraded version of 
Proposition~\ref{thetatheta}, to the latter isomorphism then yields
\begin{displaymath}
\left(\cY^{\bullet}\right)^{\oplus |W_\lambda|}\;\cong\; 
\theta^{on}_\lambda \left(\cX^\bullet\right)^{\oplus p}. 
\end{displaymath}
This implies the claim with $k=|W_\lambda|/p$.
\end{proof}

\begin{corollary}\label{critoutof}
Consider indecomposable $\cX^\bullet\in\Ob(\cD^b(\cO_0))$ 
such that we have $\theta_{w_0^\lambda}\cX^\bullet\cong 
\left(\cX^\bullet\right)^{\oplus |W_\lambda|}$. There 
is an indecomposable $\cY^\bullet\in \Ob(\cD^b(\cO_\lambda))$ 
such that  $\theta_\lambda^{out}\cY^\bullet\cong \cX^\bullet$ 
if and only if $\theta_\lambda^{on}\cX^\bullet$ contains $|W_\lambda|$ indecomposable summands.
\end{corollary}

\begin{lemma}\label{Vfaith}
Consider a module $M\in\Ob(\cO_0)$ 
such that $\theta_{w_0^\lambda}M\cong M^{\oplus |W_\lambda|}$ and 
for which $\mV$ yields an algebra isomorphism
$\mV\,:\;\;\End_{\cO_0}(M)\,\;\tilde\to\;\,\End_{\mathtt{C}}(\mV M)$.
Then $\mV_\lambda$ also induces an algebra isomorphism
$\mV_\lambda\,:\;\End_{\cO_\lambda}( \theta_\lambda^{on}M)\;\,\tilde\to\;\,
\End_{\mathtt{C}_\lambda}(\mV_\lambda\theta_\lambda^{on}M)$.
\end{lemma}

\begin{proof}
Equation~\eqref{Soergeltrans} and the faithfulness of $\theta_\lambda^{out}$ and induction imply 
that we have a commuting diagram of algebra homomorphisms:
\begin{displaymath}
\xymatrix{
\End_{\cO_0}(M^{\oplus |W_\lambda| })\ar[rr]^{\mV}&&\End_{\mathtt{C}}(\mV M^{\oplus |W_\lambda| })\\
\End_{\cO_\lambda}(\theta_\lambda^{on}M)\ar[rr]^{\mV_\lambda}\ar@{^{(}->}[u]^{\theta_\lambda^{out}}&&
\End_{\mathtt{C}_\lambda}(\mV_\lambda\theta_\lambda^{on}M)\ar@{^{(}->}[u]^{\ind_{\mathtt{C}_\lambda}^{\mathtt{C}}}\\
}
\end{displaymath}
As the upper horizontal arrow is an isomorphism by assumption, so is  $\mV_\lambda$.
\end{proof}


\begin{lemma}\label{XYZ}
Consider objects $\cX^\bullet$ of $\cD^b(\cO_\lambda)$ and 
$\cY^\bullet$ of $\cD^b(\cO^\mu_0)$ such that  we have
$\theta_\lambda^{out}\cX^\bullet\cong\imath^\mu \cY^\bullet$ in $\cD^b(\cO_0)$. 
Then there exists $\cZ^\bullet\in\cD^b(\cO_\lambda^\mu)$ such that
\begin{displaymath}
\cX^\bullet\,\cong \imath^\mu \cZ^\bullet\qquad\mbox{and}\qquad 
\cY^\bullet\,\cong\, \theta_{\lambda}^{out}\cZ^\bullet. 
\end{displaymath}
\end{lemma}

\begin{proof}
Applying $\theta_\lambda^{on}$ to 
$\theta_\lambda^{out}\cX^\bullet\cong\imath^\mu \cY^\bullet$ and using the 
commutation of inclusion and translation, implies 
\begin{displaymath}
\left(\cX^\bullet\right)^{\oplus |W_\lambda|}\;\cong\;
\imath^\mu \theta_\lambda^{on}\cY^\bullet. 
\end{displaymath}
Hence $\cX^\bullet\cong \imath^\mu\cZ^\bullet$, for some $\cZ^\bullet$. 
Applying $\theta^{out}_\lambda$ to the latter yields an isomorphism 
$\imath^\mu\cY^\bullet\cong \imath^\mu\theta_\lambda^{out}\cZ^\bullet$, 
so the conclusion follows from Corollary \ref{corisomclosed}.
\end{proof}


\section{Shuffling and projective functors}\label{secProjShuff}

Unlike twisting functors, shuffling functors do not commute with projective functors in general. 
Here we investigate their intertwining relations. We consider the principle block $\cO_0$ 
for an arbitrary reductive Lie algebra~$\fg$.

\begin{theorem}\label{simplerefl}
For two simple reflections $s,t\in W$, we have 
\begin{displaymath}
\begin{cases}
\theta_s C_t\cong C_t\theta_s, & \mbox{if $st=ts$};\\
\theta_{s} C_{st}\cong C_{st}\theta_t, &\mbox{if $s,t$ generate $S_3$}.
\end{cases} 
\end{displaymath}
\end{theorem}

The above statement is formulated in \cite[Section~11]{simple}, however, 
the proof of \cite[Section 11]{simple} is not complete (it proves only a special case of the statement). 
Below, we follow the alternative argument outlined in \cite[Remarks~11.2 and~11.4]{simple}.

\begin{corollary}\label{corforequi}
Consider $\nu\in\intdom$ such that $W_\nu$ is of type~$A$. 
For any simple reflection $s\in W_\nu\times W_\nu^\dagger$, we have
$C_{w_0^\nu}\,\theta_s\;\cong\; \theta_{s'}\, C_{w_0^\nu}$,
with $s'$ the simple reflection defined as $s'=w_0^\nu s w_0^\nu$.
\end{corollary}

Set $\cQ_{(w)}:=\add(C_wP)$, with $w\in W$,
and $P$ a projective generator for~$\cO_0$.

\begin{proposition}\label{propcommu}
{\hspace{2mm}}

\begin{enumerate}[$($i$)$]
\item\label{propcommu.1}
For $\nu\in\intdom$ and a simple reflection 
$s\in W_\nu\times W_\nu^\dagger$, we have
$C_{w_0^\nu}\theta_s\;\cong\;  \theta_{w_0^\nu sw_0^\nu} C_{w_0^\nu}$
if and only if the category $\cQ_{(w_0^\nu)}$ is stable under 
action of $\theta_{w_0^\nu sw_0^\nu}$.
\item\label{propcommu.2} 
For any $y\in W$, we have $C_{w_0}\theta_y\;\cong\;  \theta_{w_0 yw_0} C_{w_0}$. 
\end{enumerate}
\end{proposition}

\begin{remark}
{\rm 
Note that the stability of $\cQ_{(w)}$ is not an obvious property. Already for $\mathfrak{g}=\mathfrak{sl}(3)$, the category $\cQ_{(s)}$ is not stable under projective functors.
}
\end{remark}

Now we start the proofs. Define the graded lift of $C_t$ via the exact sequence 
\begin{equation}\label{deffshufflgr}
\Id\langle 1\rangle \,\overset{\rm adj}{\longrightarrow}\, \theta_t\,\to\, C_t\,\to\, 0.
\end{equation}

\begin{lemma}\label{lemF}
For a simple reflection $t$, define the functor $F_t$ as the kernel 
of the adjunction morphism in \eqref{deffshufflgr}. Then $F_t\theta_t\cong 0 \cong \theta_t F_t$.
\end{lemma}

\begin{proof}
As $\Id$ is exact, $F_t$ is left exact. Then $\theta_t T_t$ is also left 
exact and so is $T_t \theta_t$, as $\theta_t$ maps injective modules to injective modules. 
By construction, we have
\begin{displaymath}
F_tL=0\;  \mbox{ if }\;\theta_t L\not=0\quad\mbox{and}\quad  F_tL\cong L\;\mbox{ if }\; \theta_tL\cong 0.
\end{displaymath}
This implies immediately that $\theta_t F_t\cong 0$.

To prove that $F_t\theta_t$ acts trivially on all simple modules, it suffices 
to consider a simple module $L$ with $\theta_t L\not=0$. 
But then $\theta_t L$ has simple socle $L$ and the latter is not annihilated by 
$\theta_t$. This means that the adjunction morphism is injective on both
$L$ and $\theta_t L$, which implies $F_t\theta_t L=0$. The claim follows.
\end{proof}

\begin{proof}[Proof of Theorem~\ref{simplerefl}]
Assume first that $s$ and $t$ commute but are distinct, in particular we have $\theta_s\theta_t\cong\theta_{ts}\cong\theta_t\theta_s$. We can compose 
\eqref{deffshufflgr} with $\theta_s$ on both sides. This yields exact sequences
\begin{eqnarray*}
\theta_s\langle 1 \rangle\,\to \,\theta_{st}\,\to\, C_t\theta_s\,\to\,0\qquad\mbox{and}\qquad \theta_s\langle 1\rangle\,\to\, \theta_{st}\,\to\, \theta_s C_t\,\to\,0.
\end{eqnarray*}
From \cite[Theorem~3.5]{BG} and Kazhdan-Lusztig combinatorics it follows that 
the morphism $\theta_s\langle 1 \rangle\to \theta_{st}$ 
is unique, up to a non-zero scalar, implying $C_t\theta_s\cong \theta_sC_t$.

Now consider the case $s=t$. Using Lemma \ref{lemF} and the relation 
$\theta_s^2\cong\theta_s\langle 1\rangle\oplus\theta_s\langle-1\rangle$ 
we obtain two short exact sequences
\begin{eqnarray*}
&&0\to\theta_s\langle 1\rangle\to\theta_s\langle 1\rangle\oplus \theta_s\langle -1\rangle\to \theta_sC_s\to 0;\\
&&0\to\theta_s\langle 1\rangle\to\theta_s\langle 1\rangle\oplus \theta_s\langle -1\rangle\to C_s\theta_s\to 0.
\end{eqnarray*}
The only possible way to have such an injection 
$\theta_s\langle 1\rangle\hookrightarrow\theta_s\langle 1\rangle\oplus \theta_s\langle -1\rangle$ 
corresponds to $\theta_s\langle 1\rangle\,\tilde\to\,\theta_s\langle 1\rangle$. This implies indeed $\theta_sC_s\cong\theta_s\langle -1\rangle\;\cong\; C_s\theta_s. $

Finally, assume that the simple reflections $s$ and $t$ together generate a 
group isomorphic to $S_3$. Using the definitions of $C_s$ and $C_t$ via 
\eqref{deffshufflgr}, we construct a commutative diagram with exact rows and columns as follows:
\vspace{-1mm}
\begin{displaymath}
    \xymatrix{
\Id\langle 2\rangle\ar[r]\ar[d]&\theta_s\langle 1\rangle\ar[r]\ar[d]&C_s\langle 1\rangle\ar[r]\ar[d]&0\\
\theta_t\langle 1\rangle\ar[r]\ar[d]&\theta_t\theta_s\ar[r]\ar[d]&\theta_tC_s\ar[r]\ar[d]&0\\
C_t\langle 1   \rangle\ar[r]\ar[d]&C_t\theta_s\ar[r]\ar[d] & C_{st}\ar[d]\ar[r]&0\\
0&0&0
}
\end{displaymath}
Now we compose every functor in the diagram with the exact functor $\theta_s$ 
and use the above result, for the case $s=t$, in the first row of the following diagram:
\begin{displaymath}
    \xymatrix{
0\ar[r]&\theta_s\langle 2\rangle\ar[r]^{\alpha_1}\ar[d]^{f_1}&\theta_s\langle 2\rangle\oplus \theta_s\ar[r]^{\alpha_2}\ar[d]^{g_1}&\theta_s\ar[r]\ar[d]^{h_1}&0\\
&\theta_s\theta_t\langle 1\rangle\ar[r]^{\beta_1}\ar[d]^{f_2}&\theta_s\theta_t\theta_s\ar[r]^{\beta_2}\ar[d]^{g_2}&\theta_s\theta_tC_s\ar[r]\ar[d]^{h_2}&0\\
&\theta_s C_t\langle 1   \rangle\ar[r]^{\gamma_1}\ar[d]&\theta_s C_t\theta_s\ar[r]^{\gamma_2}\ar[d] &\theta_s C_{st}\ar[d]\ar[r]&0\\
&0&0&0
}
\end{displaymath}
This diagram contains a surjection $ \theta_s\theta_t\theta_s\tto\theta_sC_{st}$ 
given by  $h_2\circ\beta_2=\gamma_2\circ h_2$. The kernel of this map can 
be described as the direct sum of the image of $\beta_1$ with the preimage 
under $\beta_2$ of the image of $h_1$. Because the top row is a split 
exact sequence, the latter corresponds to the image of $\theta_s$ under 
$g_1$. This implies the exact sequence
\begin{displaymath}
\theta_s\oplus \theta_s\theta_t\langle 1\rangle\, \to\, 
\theta_s\theta_t\theta_s\,\to\,\theta_sC_{st}\,\to\, 0. 
\end{displaymath}
From Kazhdan-Lusztig combinatorics we know that 
$\theta_s\theta_t\theta_s\cong \theta_{sts}\oplus \theta_s$. 
As there is no graded morphism $\theta_s\to\theta_{sts}$ and 
the only graded morphism $\theta_s\to\theta_s$ is the identity, up to a scalar, we finally obtain an exact sequence
\begin{displaymath}
\theta_s\theta_t\langle 1\rangle\, \to\, \theta_{sts}\,\to\,\theta_sC_{st}\,\to\, 0. 
\end{displaymath}
 
Analogously we can compose $\theta_t$ with every functor in the first diagram, eventually 
yielding an exact sequence
\begin{displaymath}
\theta_s\theta_t\langle 1\rangle\, \to\, \theta_{sts}\,\to\,C_{st}\theta_t\,\to\, 0. 
\end{displaymath}
From Kazhdan-Lusztig combinatorics it, moreover, follows that there is only 
one graded morphism $\theta_s\theta_t\langle 1\rangle \to \theta_{sts}$, which concludes the proof.
\end{proof}

\begin{proof}[Proof of Corollary~\ref{corforequi}]
If $s\in W_\nu^\dagger$, the result follows immediately from Theorem~\ref{simplerefl}. 
The case $s\in W_\nu$ can be reduced to $W_\nu \cong S_n$ generated by 
$s_1,\cdots,s_n$ and $s=s_i$, for some $1\le i \le n$. 
In case $i\le n/2$, we take the reduced form
\begin{displaymath}
w_0^\nu=s_n s_{n-1} \cdots s_2 s_1s_n s_{n-1}\cdots s_3s_2\cdots   s_ns_{n-1} s_{n-2} s_n s_{n-1} s_n;
\end{displaymath}
if $i\ge n/2$, we take
\begin{displaymath}
w_0^\nu=s_1 s_2\cdots s_{n-1}s_n s_1s_2\cdots s_{n-2}s_{n-1}\cdots s_1 s_2 s_3 s_1 s_2 s_1.
\end{displaymath}
In both cases, apply Theorem~\ref{simplerefl} repeatedly, 
giving the result with $s'=s_{n-i+1}$.
\end{proof}

\begin{lemma}\label{LEMeqapr19-1}
For any simple reflection $t\in W$, we have
$\mathbb{V} C_{t}\cong {}^{t}\mathtt{C}\otimes_{\mathtt{C}}\mathbb{V}$.
\end{lemma}

\begin{proof}
Since both sides are right exact, it is enough to check the isomorphism on 
the category of projective modules in $\cO_0$. 
\cite[Theorem~4.9]{Backelin} implies that there is a unique injective bimodule homomorphism 
$\mathtt{C}\hookrightarrow \mathtt{C}\otimes_{\mathtt{C}^t}\mathtt{C}$. We claim that the cokernel of
this maps is isomorphic to ${}^{t}\mathtt{C}$. Indeed, this cokernel is isomorphic to $\mathtt{C}$ 
both as a left and as a right $\mathtt{C}$-module. As a $\mathtt{C}^t\text{-}\mathtt{C}^t$-bimodule, 
we obviously get a decomposition of the cokernel into a direct sum of two copies
of~$\mathtt{C}^t$. To determine the action of the remaining generator, it is enough to compute the
rank two case. In that case the statement is easily checked by a direct computation.
\end{proof}

\begin{corollary}\label{Visom}
For $\nu\in\intdom$ and a simple reflection $s\in W_\nu\times W_\nu^\dagger$, we have
\begin{displaymath}
\mV C_{w_0^\nu}\theta_s\;\cong\;  \mV \theta_{w_0^\nu sw_0^\nu} C_{w_0^\nu}. 
\end{displaymath}
\end{corollary}

\begin{proof}
By Subsection \ref{sscentre}, we have 
\begin{displaymath}
\mathbb{V} \theta_s\cong \mathtt{C}\otimes_{\mathtt{C}^s}\mathtt{C}\otimes_{\mathtt{C}}\mathbb{V} 
\,\,\, \text{ and }\,\,\,
\mathbb{V} \theta_{w_0^\nu s w_0^\nu}\cong \mathtt{C}\otimes_{\mathtt{C}^{w_0^\nu s w_0^\nu}}\mathtt{C}\otimes_{\mathtt{C}}\mathbb{V}.  
\end{displaymath}

By induction and the choice of a reduced expression for $w_0^\nu$, 
Lemma~\ref{LEMeqapr19-1} implies
$\mathbb{V} C_{w_0^\nu}\cong {}^{w_0^\nu}\mathtt{C}\otimes_{\mathtt{C}}\mathbb{V}$.
A straightforward computation gives an isomorphism
\begin{displaymath}
\mathtt{C}\otimes_{\mathtt{C}^{w_0^\nu s w_0^\nu}}{}^{w_0^\nu}\mathtt{C}\cong {}^{w_0^\nu}\mathtt{C}\otimes_{\mathtt{C}^s}\mathtt{C},
\end{displaymath}
of~$\mathtt{C}\text{-}\mathtt{C}$ bimodules. Combining the three 
isomorphisms yields the claim.
\end{proof}

\begin{lemma}\label{catQ}
Soergel's combinatorial functor $\mV$ is full and faithful when 
restricted to the category $\cQ_{(w)}$, for any $w\in W$.
\end{lemma}

\begin{proof}
As $\cL C_w$ is an auto-equivalence of $\cD^b(\cO_0)$, 
see equation~\eqref{dershuff0}, $C_w$ provides an isomorphism 
from $\Hom_{\cO_0}(P, P')$ to $\Hom_{\cO_0}(C_wP,C_wP')$ 
for any two projective modules $P,P'$ in $\cO_0$. The claim 
hence reduces to the statement that the functor $\mV C_w$ is 
full and faithful on the category of projective modules. 
By Lemma~\ref{LEMeqapr19-1} this reduces to \cite[Struktursatz~9]{SoergelD}.
\end{proof}

\begin{proof}[Proof of Proposition~\ref{propcommu}]
We prove claim~\eqref{propcommu.1} first. It is clear that if the proposed 
equivalence of functors holds, then the fact that the category 
$\cQ_{(w_0^\nu)}$ is stable under $\theta_{w_0^\nu sw_0^\nu}$ 
follows from the fact that $\theta_s$ preserves the category of 
projective modules. Hence we focus on the other direction of the claim.

The two composed functors in the proposition are right exact, as shuffling 
functors are right exact and projective functors are exact functors 
mapping projective modules to projective modules. Hence it suffices 
to prove the isomorphism as functors restricted to $\cQ_{(e)}$, the full subcategory 
of projective modules in~$\cO_0$. 

By assumption both functors restrict to functors between $\cQ_{(e)}$ 
and $\cQ_{(w_0^\nu)}$. Hence, the combination of Lemma~\ref{catQ} and 
Corollary~\ref{Visom} concludes the proof of claim~\eqref{propcommu.1}.

Now we consider claim~\eqref{propcommu.1} for $w_0^\nu=w_0$. As the category 
$\cQ_{(w_0)}$ is the category of q.h. tilting modules, see e.g. 
\cite[Section~4.2]{prinjective}, it is stable under the action of all 
projective functors. Therefore claim~\eqref{propcommu.2} holds for $y=s$ a simple reflection. 
The full statement then follows by induction on the length of~$y$ and \cite[Equation~(1)]{MR2366357}. 
\end{proof}

We have the following application of Proposition~\ref{propcommu}(2).

\begin{corollary}\label{shuffstandards}
Consider $\Delta^\mu(x\cdot\lambda)^\bullet$ as an object of $\cD^b(\cO_\lambda)$, then
\begin{equation*}
\label{shuffeq}\cL C_{w_0}\theta_\lambda^{out} 
\Delta^\mu(x\cdot\lambda)^\bullet\,\;\cong\;\, 
\theta_{\widehat\lambda}^{out}\nabla^{\mu}
(w_0^\mu x w_0^\lambda w_0 \cdot\widehat\lambda)^\bullet [l(w_0^\mu)].
\end{equation*}
Moreover, the same equation holds considering $\Delta^\mu(x\cdot\lambda)^\bullet$ as an 
object of $\cD^b(\cO_\lambda^\mu)$ and $\cL C_{w_0}$ the endofunctor 
on $\cD^b(\cO_0^\mu)$ defined in Proposition~\ref{commdiagshuff}.
\end{corollary}

\begin{proof}
By Theorem~\ref{thmDeltaon} and Proposition~\ref{propcommu}(2), 
we have 
\begin{displaymath}
\cL C_{w_0} \theta_\lambda^{out} \Delta^\mu (x\cdot\lambda)^\bullet\,\cong\, 
\cL C_{w_0} \theta_{w_0^\lambda} \Delta^\mu (x)^\bullet\,\cong\,
\theta_{w_0^{\widehat\lambda}} \cL C_{w_0} \Delta^\mu (x)^\bullet. 
\end{displaymath}
Applying then \cite[Proposition~4.4(1)]{prinjective}, gives
\begin{displaymath}
\cL C_{w_0} \theta_\lambda^{out} \Delta^\mu 
(x\cdot\lambda)^\bullet\,\cong\, \theta_{w_0^{\widehat\lambda}} 
\nabla^\mu (w_0^\mu x w_0)^\bullet[l(w_0^\mu)]. 
\end{displaymath}
The result now follows by applying Theorem~\ref{thmDeltaon}.
The reformulation inside $\cD^b(\cO_0^\mu)$ follows from the 
commuting diagram in Proposition~\ref{commdiagshuff} and Corollary \ref{corisomclosed}.
\end{proof}

Finally, we derive useful expression for regular parabolic q.h. tilting modules.

\begin{proposition}\label{tiltingtheta}
For any $x\in X^\mu$, we have $T^\mu(x)\cong \theta_{w_0w_0^\mu x} L(w_0^\mu w_0)$.
\end{proposition}

\begin{proof}
As $L(w_0^\mu w_0)\cong \Delta^\mu(w_0^\mu w_0)\cong \nabla^\mu(w_0^\mu w_0)$ 
is a tilting module in $\cO^\mu_0$ and projective functors preserve tilting modules, 
it follows that $\theta_{w_0w_0^\mu x} L(w_0^\mu w_0)$ is a tilting module in $\cO^\mu_0$.
From Kazhdan-Lusztig combinatorics, see e.g. \cite[Equations (2.1) and~(2.2)]{MR2563180}, 
it follows by induction on the length $l(w_0 w_0^\mu x)$ that the highest weight of 
$\theta_{w_0w_0^\mu x} L(w_0^\mu w_0)$ is $x\cdot 0$. 

It remains to prove the indecomposability of $\theta_{w_0w_0^\mu x} L(w_0^\mu w_0)$. 
By \cite[Theorem~16]{SHPO2}, the 
Koszul-Ringel self-duality of $\cO_0$ maps 
$\theta_{w_0w_0^\mu x} L(w_0^\mu w_0)$ to a module isomorphic to 
$\theta_{w_0^{\widehat\mu}}L(w_0x^{-1})$. This module is the translation through the wall 
of a simple module, which has simple top and hence is indecomposable. 
\end{proof}


\section{Construction of derived equivalences}\label{secConstr}

\begin{theorem}\label{maindereq}
Let $\fg$ be a reductive Lie algebra, $\lambda,\lambda',\mu,\mu'\in\intdom$ and 
$\nu_1,\nu_2\in\intdom$, such that $W_{\nu_1}\cong S_{n_1}$ 
and $W_{\nu_2}\cong S_{n_2}$, for some $n_1,n_2\in\mN$. 
Assume that  
\begin{displaymath}
W_\lambda=G_1\times G_2,\quad W_{\lambda'}=G'_1\times G'_2,\quad 
W_\mu=H_1\times H_2,\quad W_{\mu'}= H_1'\times H_2'\; \mbox{ with}
\end{displaymath}
\begin{itemize}
\item $G_2=G'_2\subset W_{\nu_1}^\dagger$ and 
$H_2 =H_2'\subset W_{\nu_2}^\dagger$ ;
\item $G_1\cong G_1'$ and both are subgroups of $W_{\nu_1}$;
\item $H_1\cong H_1'$ and both are subgroups of $W_{\nu_2}$.
\end{itemize}
Then $A_\lambda^\mu$ and $A_{\lambda'}^{\mu'}$ are gradable derived 
equivalent, so, in particular, we have equivalences of triangulated categories 
\begin{displaymath}
\cD^b(\cO^\mu_\lambda)\;\cong\; \cD^b(\cO^{\mu'}_{\lambda'})
\qquad\mbox{and}\qquad \cD^b({}^{\mZ}\cO^\mu_\lambda)\;\cong\; 
\cD^b({}^{\mZ}\cO^{\mu'}_{\lambda'}). 
\end{displaymath}
\end{theorem}

For $\mathfrak{sl}(n)$, the formulation of the theorem simplifies substantially. 
In this case, without loss of generality, we can take $\nu_1=\nu_2=-\rho$, so $G_2=G_2'=H_2=H_2'=\{e\}$. 
Then we obtain precisely Theorem~C.

The remainder of this section~is devoted to the proof of Theorem~\ref{maindereq}.

\begin{lemma}\label{equivlambda}
For $\lambda\in\intdom$, assume there is $\nu\in\intdom$ such that 
$w_0^\lambda\in W_{\nu}\times W_{\nu}^\dagger$ 
where $W_\nu$ is of type~$A$. Then there is $\lambda'\in\intdom$
such that $w_0^{\lambda'}=w_0^\nu w_0^\lambda w_0^\nu$ and a triangulated 
equivalence $F_{\lambda\lambda'}:\cD^b(\cO_\lambda)\to \cD^b(\cO_{\lambda'})$ 
leading to a commutative diagram of functors
\vspace{-1mm}
\begin{displaymath}
    \xymatrix{
\cD^b(\cO_0)\ar[rr]^{\cL C_{w_0^\nu}}&&\cD^b(\cO_0)\\
\cD^b(\cO_\lambda)\ar[u]^{\theta_\lambda^{out}}\ar[rr]^{F_{\lambda\lambda'}}
&&\cD^b(\cO_{\lambda'})\ar[u]^{\theta_{\lambda'}^{out}}   }
\end{displaymath}
\end{lemma}

\begin{proof}
For $x\in X_\lambda$, set $S_x:=C_{w_0^\nu}\theta_\lambda^{out}P(x\cdot\lambda)$. Corollary~\ref{corforequi} yields 
\begin{equation}\label{tiltingthro}
\theta_{w_0^{\lambda'}} S_x\,\cong\, S_x^{\oplus |W_{\lambda'}|}.
\end{equation}
By Lemma~\ref{catQ}, $\End_{\cO_0}(S_x)\cong\End_{\mathtt{C}}(\mV S_x).$
Hence \eqref{tiltingthro} and Lemma~\ref{Vfaith} imply
\begin{displaymath}
\End_{\cO_{\lambda'}}(\theta_{\lambda'}^{on}S_x)\;\cong\;
\End_{\mathtt{C}_{\lambda'}}(\mV_{\lambda'}\theta_{\lambda'}^{on} S_x). 
\end{displaymath}
The algebra can then be computed by using equation~\eqref{Soergeltrans} 
and Lemma~\ref{LEMeqapr19-1},
\begin{displaymath}
\End_{\cO_{\lambda'}}(\theta_{\lambda'}^{on}S_x)\cong 
\End_{\mathtt C_{\lambda'}}(\res^{\mathtt C}_{\mathtt C_{\lambda'}}
{}^{w_0^\nu}\mathtt{C}\otimes_{\mathtt{C}}\mathtt{C}
\otimes_{\mathtt{C}_\lambda}\mV_\lambda P(x\cdot\lambda)). 
\end{displaymath}
Under the isomorphism $\mathtt{C}_{\lambda}\cong\mathtt{C}_{\lambda'}$, 
we can identify the modules 
\begin{displaymath}
\res^{\mathtt C}_{\mathtt C_{\lambda'}} {}^{w_0^\nu}
\mathtt{C}\otimes_{\mathtt{C}}\mathtt{C}
\otimes_{\mathtt{C}_\lambda}\mV_\lambda P(x\cdot\lambda)
\;\mbox{ and }\;\;\;\res^{\mathtt C}_{\mathtt C_{\lambda}} 
\mathtt{C}\otimes_{\mathtt{C}_\lambda}\mV_\lambda P(x\cdot\lambda). 
\end{displaymath}
Therefore we obtain 
\begin{equation}\label{eqendalg}
\End_{\cO_{\lambda'}}(\theta_{\lambda'}^{on}S_x)\,\cong\,
\End_{\cO_\lambda}(P(x\cdot\lambda)^{\oplus |W_\lambda|}).
\end{equation}
This formula implies that $\theta^{on}_{\lambda'}S_x$ decomposes into 
$|W_{\lambda'}|$ indecomposable direct summands. 
Equation~\eqref{tiltingthro} and Corollary~\ref{critoutof} thus imply 
that there is an indecomposable $T_x\in{\cO_{\lambda'}}$ such that 
$S_x\cong \theta_{\lambda'}^{out}T_x$. Then we define 
\begin{displaymath}
S=\bigoplus_{x\in X_\lambda}S_x\quad\text{ and }\quad T=\bigoplus_{x\in X_\lambda} T_x. 
\end{displaymath}
Now equation \eqref{eqendalg} can be rewritten as
\begin{displaymath}
\End_{\cO_{\lambda'}}(T_x^{\oplus |W_\lambda|})\,\cong\,
\End_{\cO_\lambda}(P(x\cdot\lambda)^{\oplus |W_\lambda|}). 
\end{displaymath}
The same calculation for $P_\lambda$, 
rather than $P(x\cdot\lambda)$, implies
\begin{equation}\label{endalg}
\End_{\cO_{\lambda'}}(T)\cong \End_{\cO_\lambda}(P_\lambda)=A_\lambda.
\end{equation}
As $\cL C_{w_0^\nu}$ is an auto-equivalence of $\cD^b(\cO_0)$, 
the module $S$ does not have self-extensions. As 
$\theta_{\lambda'}^{out}:\cD^b(\cO_{\lambda'})\to\cD^b(\cO_0)$ 
is faithful, by construction, the module $T$ then satisfies 
\begin{equation}\label{noselfext}
\Ext^i_{\cO_{\lambda'}}(T,T)=0,\quad\mbox{for}\;\, i>0.
\end{equation}

For any projective module $P$ in $\cO_0$, we claim that 
$\theta_{\lambda'}^{on} C_{w_0^\nu}P$ is a module in add$(T)$. 
Indeed, by Corollary~\ref{corforequi} we obtain
\begin{displaymath}
\left(\theta_{\lambda'}^{on}C_{w_0^\nu}P\right)^{\oplus |W_{\lambda'}|}\;
\cong\; \theta_{\lambda'}^{on}C_{w_0^\nu}\theta_{w_0^\lambda}P,
\end{displaymath}
where $C_{w_0^\nu}\theta_{w_0^\lambda}P$ is clearly a module in add$(S)$ 
and $\theta_{\lambda'}^{on}S\cong T^{\oplus |W_\lambda|}$. 

As $\cL C_{w_0^\nu}$ is an auto-equivalence of $\cD^b(\cO_0)$, 
\cite[Theorem~6.4]{Rickard} implies that $C_{w_0^\nu} P_0$ is a 
generalised tilting module. Hence, $P_0$ has a coresolution by modules 
in add$(C_{w_0^\nu}P_0)$. By the above paragraph, applying the exact 
functor $\theta_{\lambda'}^{on}$ leads to a coresolution of 
$\theta_{\lambda'}^{on}P_0$ by modules in add$(T)$. As $T$ has no 
self-extensions by equation \eqref{noselfext}, this immediately 
implies that every indecomposable summand in $\theta_{\lambda'}^{on}P_0$ 
has such a coresolution. Thus, in particular, there is some $r\in \mN$ 
for which there is an exact sequence
\begin{equation}\label{codim}
0\to P_{\lambda'}\to T_0\to T_1\to\cdots T_{r-1}\to T_r\to 0,
\end{equation}
where $T_i\in$ add$(T)$.

Equations \eqref{noselfext} and \eqref{codim} imply that $T$ is a 
generalised tilting module, 
or a tilting complex (contained in one position) in the sense 
of~\cite{Rickard}. The equivalence then follows from equation 
\eqref{endalg} and \cite[Theorem~III.2.10]{Happel} or \cite[Theorem~6.4]{Rickard}. 

Finally, we prove the existence of the commutative diagram. 
It suffices to prove that 
$\cL C_{w_0^\nu}\circ\theta_{\lambda}^{out}\cong 
\theta_{\lambda'}^{out}\circ F_{\lambda\lambda'}$, restricted as functors 
on the category of projective modules in $\cO_\lambda$, as both are 
triangulated functors. By construction, the functor $F_{\lambda\lambda'}$ 
acts on the category of projective modules by mapping $P(x\cdot\lambda)$ 
to $T_x$ and its action on morphisms corresponds to the algebra 
isomorphism determined in equation \eqref{endalg}. The equivalence 
of the functors $\cL C_{w_0^\nu}\circ\theta_{\lambda}^{out}$ and 
$ \theta_{\lambda'}^{out}\circ F_{\lambda\lambda'}$ acting between 
add$(P_\lambda)$ and add$(S)$, thus follows from construction of the isomorphism \eqref{endalg}.
\end{proof}

\begin{lemma}\label{extrlemma}
In the setup of Lemma \ref{equivlambda}, any object $\cX^\bullet$ in 
$\cD^b(\cO_\lambda)$ satisfies
\begin{displaymath}
\cL Z^\mu\imath^\mu F_{\lambda\lambda'}\cX^\bullet\;\cong\;
F_{\lambda\lambda'}\cL Z^\mu\imath^\mu\cX^\bullet. 
\end{displaymath}
\end{lemma}

\begin{proof}
According to Corollary \ref{corisomclosed}, it is sufficient to prove that
\begin{displaymath}
\theta_{\lambda'}^{out}\cL Z^\mu\imath^\mu F_{\lambda\lambda'}\cX^\bullet\;\cong\;\theta_{\lambda'}^{out}F_{\lambda\lambda'}\cL Z^\mu\imath^\mu\cX^\bullet. 
\end{displaymath}
That this is true follows from the fact that Zuckerman and inclusion 
functors commute with translation functors (and hence also with 
shuffling functors), see e.g. \cite[Lemma 2.6(a)]{Back} and the 
diagram in Lemma \ref{equivlambda}.
\end{proof}

\begin{lemma}\label{equivpar}
Consider $\lambda,\nu$ as in Lemma~\ref{equivlambda} and an arbitrary 
$\mu\in\intdom$. There is a an equivalence 
$F_{\lambda\lambda'}^{\mu}:\;\cD^b(\cO^\mu_\lambda)\,\to\,\cD^b(\cO_{\lambda'}^\mu)$ 
of triangulated categories which admits a commutative diagram of functors
\vspace{-1mm}
\begin{displaymath}
\xymatrix{
\cD^b(\cO_\lambda)\ar[rr]^{F_{\lambda\lambda'}}&&\cD^b(\cO_{\lambda'})\\
\cD^b(\cO^\mu_\lambda)\ar[u]^{\imath^\mu}\ar[rr]^{F^\mu_{\lambda\lambda'}}
&&\cD^b(\cO^\mu_{\lambda'})\ar[u]^{\imath^\mu}.   }
\end{displaymath}
\end{lemma}

\begin{proof}
Consider the minimal projective generator $P_\lambda^\mu$ of $\cO_\lambda^\mu$ in \eqref{progen}. 
By Lemma~\ref{equivlambda}, there is a complex 
$\cS^\bullet:=F_{\lambda\lambda'}\,\imath^\mu\, P^\mu_\lambda$ in $\cD^b(\cO_{\lambda'})$
with
\begin{equation}\label{EndSA}
\End_{\cD^b(\cO_{\lambda'})}(\cS^\bullet)\cong \End_{\cD^b(\cO_{\lambda})}
(\imath^\mu P^\mu_\lambda)\cong\End_{\cD^b(\cO^\mu_{\lambda})}( P^\mu_\lambda)\cong A^\mu_\lambda, 
\end{equation}
where the latter isomorphism follows from Corollary \ref{imathend}. Lemma~\ref{equivlambda} and Proposition \ref{commdiagshuff} yield
$\theta_{\lambda'}^{out}\cS^\bullet\;\cong\; \imath^\mu \cL 
C_{w_0^\nu}\theta_\lambda^{out}P_\lambda^\mu.$
Lemma \ref{XYZ} thus implies the existence of 
$\cT^\bullet\in \cD^b(\cO_\lambda^\mu)$ such that
\begin{equation}
\label{defTpar}\imath^\mu\cT^\bullet \;\cong \;\cS^\bullet\qquad\mbox{and}\qquad \theta_{\lambda'}^{out}\cT^\bullet\cong\cL C_{w_0^\nu}\theta_\lambda^{out}P_\lambda^\mu.
\end{equation}
The second property in \eqref{defTpar} and the faithfulness of 
$\theta_{\lambda'}^{out}$ give an injection
\begin{displaymath}
\Hom_{\cD^b(\cO^\mu_{\lambda'})}(\cT^\bullet,\cT^\bullet[k])
\hookrightarrow\Hom_{\cD^b(\cO^\mu_0)}(\cL C_{w_0^\nu}
\theta_\lambda^{out}P_\lambda^\mu,\cL C_{w_0^\nu}\theta_\lambda^{out}P_\lambda^\mu[k]).  
\end{displaymath}
The right-hand side is zero, if $k\not=0$, by Proposition \ref{commdiagshuff}. So we find
\begin{equation}
\label{tilteqpar1}
\Hom_{\cD^b(\cO^\mu_{\lambda'})}(\cT^\bullet,\cT^\bullet[k])=0 \qquad \mbox{if}\quad k\not=0.
\end{equation}
The first property in equation \eqref{defTpar} and equation \eqref{tilteqpar1} 
allow us to apply Corollary \ref{imathend} and deduce, using equation \eqref{EndSA}, that
\begin{equation}
\label{tilteqpar2}
\End_{\cD^b(\cO^\mu_{\lambda'})}(\cT^\bullet)\;\cong\;\End_{\cD^b(\cO_{\lambda'})}(\cS^\bullet)\;\cong\;A^\mu_\lambda.
\end{equation}
By equations \eqref{tilteqpar1} and \eqref{tilteqpar2}, the result would 
follow from \cite[Theorem~6.4]{Rickard}, if we could show that add$(\cT^\bullet)$ generates 
$\cD^b(\cO_\lambda^\mu)$ as a triangulated category.

Now we focus on the latter statement. Remark~\ref{Zinongrad} and 
Lemma~\ref{extrlemma} yield
\begin{displaymath}
\bigoplus_{j\in\mN}\cL Z^\mu F_{\lambda\lambda'}
(P_\lambda^\bullet)^{\oplus c_j}[j]\;\,\cong\;\, \cL Z^\mu 
F_{\lambda\lambda'}\imath^\mu \cL Z^\mu P_\lambda^\bullet .
\end{displaymath} 
As, for every $x\in X_\lambda$,
\begin{displaymath}
\cL Z^\mu P(x\cdot\lambda)^\bullet\cong\begin{cases}
P^\mu(x\cdot\lambda)^\bullet,&\mbox{if }x\in X^\mu_\lambda;
\\0,&\mbox{if }x\not\in X^\mu_\lambda;\end{cases}
\end{displaymath}
we find
\begin{displaymath}
\bigoplus_{j\in\mN}\cL Z^\mu F_{\lambda\lambda'}
(P_\lambda^\bullet)^{\oplus c_j}[j]\;\,\cong\;\, \cL Z^\mu 
\imath^\mu \cT^\bullet\;\, \cong\;\, \bigoplus_{j\in\mN}(\cT^\bullet)^{\oplus c_j}[j].
\end{displaymath}
This implies that
$\cL Z^\mu F_{\lambda\lambda'} P_\lambda^\bullet\,\in\, {\rm add}(\cT^\bullet)$.
The fact that $F_{\lambda\lambda'}P_\lambda^\bullet$ generates $\cD^b(\cO_{\lambda'})$ 
as a triangulated category (which is an immediate consequence of Lemma \ref{equivlambda}) hence implies that ${\rm add}(\cT^\bullet)$ generates 
$\cD^b(\cO^\mu_{\lambda'})$ as a triangulated category.
The existence of the commuting diagram then follows by the same arguments as in the 
proof of Lemma~\ref{equivlambda}.
\end{proof}

\begin{lemma}\label{equivpargrad2}
The functor $F_{\lambda\lambda'}^{\mu}$ from Lemma~\ref{equivpar} admits a 
graded lift, yielding an equivalence of triangulated categories
\begin{displaymath}
\widetilde{F}_{\lambda\lambda'}^{\mu}:\;\cD^b({}^{\mZ}\cO^\mu_\lambda)\,
\to\,\cD^b({}^{\mZ}\cO_{\lambda'}^\mu). 
\end{displaymath}
\end{lemma}

\begin{proof}
As translation, shuffling and inclusion functors admit graded lifts, it follows easily that so does $F^\mu_{\lambda,\lambda'}$.

We can switch the roles of $\lambda$ and $\lambda'$ and 
construct a functor $F^\mu_{\lambda'\lambda}$. We define the (gradable) 
functor $G^\mu_{\lambda'\lambda}=\dd F_{\lambda'\lambda}\dd$. 
Then $G^\mu_{\lambda'\lambda}$ admits commutative diagrams with 
$\dd\cL C_{w_0^\nu}\dd$, analogous to the ones in 
Lemmata~\ref{equivlambda} and~\ref{equivpar} and is also gradable. 
By construction, $F^\mu_{\lambda\lambda'}G^\mu_{\lambda'\lambda}$
and, respectively, $G^\mu_{\lambda'\lambda}F^\mu_{\lambda\lambda'}$,
are isomorphic to the respective identity functors when
restricted to the categories of projective modules in $\cO^\mu_{\lambda'}$ 
and $\cO^\mu_\lambda$, respectively. 
Hence $G^\mu_{\lambda'\lambda}$ is an inverse to~$F^\mu_{\lambda'\lambda}$. 
The result thus follows from Proposition~\ref{liftequiv}.
\end{proof}

\begin{lemma}\label{maindereqLem}
Consider  $\lambda,\lambda',\mu\in\intdom$ 
and $\nu_1\in\intdom$ such that $W_{\nu_1}\cong S_{n}$, for 
some $n\in\mN$. Assume that $W_\lambda=G_1\times G_2$ and 
$W_{\lambda'}=G'_1\times G'_2$, where 
\begin{itemize}
\item $G_2=G'_2$ is a subgroup of $W_{\nu_1}^\dagger$;
\item $G_1\cong G_1'$ and both are subgroups of $W_{\nu_1}$.
\end{itemize}
Then $A^\mu_\lambda$ and $A^\mu_{\lambda'}$ are gradable derived equivalent. In particular, we have equivalences of 
triangulated categories
\begin{displaymath}
\cD^b(\cO^\mu_\lambda)\;\cong\; \cD^b(\cO^{\mu}_{\lambda'})
\qquad\mbox{and}\qquad \cD^b({}^{\mZ}\cO^\mu_\lambda)\;\cong\; \cD^b({}^{\mZ}\cO^{\mu}_{\lambda'}). 
\end{displaymath}
\end{lemma}
 
\begin{proof}
There are $k\in\mN$ and $p_i\in\mN$, for $1\le i\le k$, 
such that $p_1+\dots+p_k=n$ and 
\begin{displaymath} 
G_1=S_{p_1}\times S_{p_2}\times \cdots \times S_{p_k}
\end{displaymath}
as a subgroup of $S_n$. Then $G_2$, as a subgroup of $S_n$, is equal to
\begin{displaymath}
G_2=S_{p_{\sigma(1)}}\times S_{p_{\sigma(2)}}\times \cdots \times S_{p_{\sigma(k)}}, 
\end{displaymath}
for some $\sigma\in S_k$. It suffices to prove the lemma for $\sigma$ 
such that it only interchanges two neighbors $(j,j+1)$ for some $1\le j<k$. 
Then we choose an integral dominant $\nu$ such that $W_\nu$ is equal to 
the subgroup of $W_{\nu_1}\subset W$ defined as
\begin{displaymath}
\stackrel{p_1+p_2+\dots+p_{j-1}}{\overbrace{S_1\times S_1\times 
\cdots \times S_1}}\times S_{p_j+p_{j+1}}\times 
\stackrel{p_{j+2}+\dots+p_k}{\overbrace{S_1\times \cdots S_1}}.
\end{displaymath}
For the case we consider, this leads to 
$w_0^{\lambda'}=w_0^\nu w_0^\lambda w_0^\nu$ and the equivalences 
follow from Lemma~\ref{equivpar}, Lemma~\ref{equivpargrad2} and 
Definition~\ref{defgradeq}.
\end{proof}

\begin{proof}[Proof of Theorem~\ref{maindereq}]
To prove the theorem it suffices to restrict to either $\lambda=\lambda'$ 
or $\mu=\mu'$ as the full statement follows by composition of the two. 
The case $\mu=\mu'$ is precisely Lemma~\ref{maindereqLem}.

The graded equivalence for $\lambda=\lambda'$ follows immediately from 
composing Lemma~\ref{maindereqLem} with the equivalence in 
equation~\eqref{KoszuleqO}. The fact that the equivalence then descends 
to the non-graded categories follows from Proposition~\ref{descequiv} and equation \eqref{KdFshift}.
\end{proof}

Finally, we note how the results in this section lead to a proof of 
Theorem~D. Part (ii) follows immediately from Lemmata~\ref{equivlambda}, 
\ref{equivpar} and~\ref{equivpargrad2}. Part (i) then follows from 
Koszul duality as in the proof of Theorem \ref{maindereq}. The link
with derived twisting functors follows from \cite[Theorem 39]{MOS}.


\section{The classification for type~$A$}\label{secClass}

In this section we prove the following theorem, which implies Theorem~B.

\begin{theorem}\label{class3}
Consider two Lie algebras $\fg,\fg'$ of type~$A$ with fixed Borel 
subalgebras $\fb,\fb'$, classes $\Lambda\in \fh^\ast/\Lambda_{\mathrm{int}}$ and 
$\Lambda'\in(\fh')^\ast/\Lambda_{\mathrm{int}}'$ with the corresponding integral Weyl 
groups $W_{\Lambda},W'_{\Lambda'}$ and two dominant weights 
$\lambda\in \Lambda$, $\lambda'\in\Lambda'$. Then the following statements are equivalent:
\begin{enumerate}[$($I$)$]
\item\label{class3.1} There is a gradable equivalence 
$\cD^b(\cO_\lambda(\fg,\fb))\cong \cD^b(\cO_{\lambda'}(\fg',\fb'))$.
\item\label{class3.2} There is a graded algebra isomorphism 
$\cZ(\cO_\lambda(\fg,\fb))\cong \cZ(\cO_{\lambda'}(\fg',\fb'))$.
\item\label{class3.3} For some decompositions 
\begin{displaymath}
W_{\Lambda}\cong X_1\times X_2\times \dots\times X_k\quad\text{ and }\quad
W'_{\Lambda'}\cong X'_1\times X'_2\times \dots\times X'_m
\end{displaymath}
into products of irreducible Weyl groups, we have $k=m$ and there is a permutation
$\varphi$ on $\{1,2,\dots,k\}$ such that $W_{\Lambda,\lambda}\cap X_i\cong
W_{\Lambda',\lambda'}\cap X'_{\varphi(i)}$ and $X_i\cong X'_{\varphi(i)}$ for all $i=1,2,\dots,k$.
\end{enumerate}
\end{theorem}

Note that, as we only consider type $A$ in this section, we do not need to distinguish between isomorphisms of Coxeter groups and isomorphisms of Coxeter systems. Before proving Theorem \ref{class3} we need the following preparatory lemma.

\begin{lemma}\label{lemCox}
Consider Coxeter groups $W$ and $W'$ of type~$A$ with respective 
Young subgroups $X$ and $X'$. There is a graded algebra isomorphism 
$\mathtt{C}(W)^X\cong \mathtt{C}(W')^{X'}$ if and only if,
for some decompositions 
\begin{displaymath}
W\cong X_1\times X_2\times \dots\times X_k\quad\text{ and }\quad
W'\cong X'_1\times X'_2\times \dots\times X'_m
\end{displaymath}
into products of irreducible Weyl groups, we have $k=m$ and there is a permutation
$\varphi$ on $\{1,2,\dots,k\}$ such that $X\cap X_i\cong
X'\cap X'_{\varphi(i)}$ and $X_i\cong X'_{\varphi(i)}$ for all $i=1,2,\dots,k$.
\end{lemma}

\begin{proof}
First we assume that $W$ and $W'$ are simple. Without loss of 
generality, we assume that $W\cong S_n$ and $W'\cong S_m$ with 
$m\ge n$. Denote the compositions of $n$ and $m$, corresponding
to $X$ and $X'$, by $(p_1,\cdots,p_k)$ and 
$(q_1,\cdots ,q_l)$, respectively. A graded algebra isomorphism implies, in 
particular, a graded vector space isomorphism, so equation 
\eqref{HilPoin} implies that 
\begin{displaymath}
\prod_{j=1}^l\prod_{a=1}^{q_j}(1-z^a)\;=\;\prod_{j=1}^k
\prod_{b=1}^{p_j}(1-z^b)\prod_{s=n+1}^m(1-z^s)
\end{displaymath}
must be equal in $\mC[z]$. If $m>n$, then the irreducible 
factorisation of the right-hand side contains a term 
$(1-\exp(\frac{2\pi i}{m})z)$, with $i^2=-1$, which never occurs in the factorisation 
of the left-hand side. Hence $m=n$, or $W\cong W'$. The same 
argument can now be used if the maximum of $\{p_j,1\le j\le k\}$ 
were larger than the maximum of $\{q_j,1\le j\le l\}$. By similar reasoning and induction 
it then follows that the compositions are the same up to ordering 
and hence we find $X\cong X'$.

Now consider an arbitrary Coxeter group $W$ of type~$A$. It follows 
easily from the definitions that, if we consider 
the decomposition $W\cong W_1\times W_2\times\cdots \times W_k$, 
for some $k$, such that each $W_i$ is irreducible and set 
$X_i=X\cap W_i$, then we have 
\begin{displaymath}
\mathtt{C}(W)^X\cong \mathtt{C}(W_1)^{X_1}\oplus 
\mathtt{C}(W_2)^{X_2}\oplus \cdots \oplus\mathtt{C}(W_k)^{X_k}. 
\end{displaymath}
The general result now follows.
\end{proof}

\begin{proof}[Proof of Theorem~\ref{class3}]
By \cite[Theorem~11]{SoergelD}, we can take reductive Lie 
algebras $\tilde\fg$ and $\tilde\fg'$ for which there are 
Borel subalgebras $\tilde\fb$ and $\tilde\fb'$, Weyl groups 
$W$ and $\tilde W$, and {\em integral} dominant weights 
$\tilde\lambda$ and $\tilde\lambda'$ such that 
\begin{enumerate}[$($a$)$]
\item\label{property.1} 
$\cO_\lambda(\fg,\fb)\cong \cO_{\tilde\lambda}(\tilde\fg,\tilde\fb)$ 
and $\cO_{\lambda'}(\fg',\fb')\cong \cO_{\tilde\lambda'}(\tilde\fg',\tilde\fb')$;
\item\label{property.2} 
$W_{\Lambda}\cong \tilde W$ and $W'_{\Lambda'}\cong \tilde W'$;
\item\label{property.3} 
$W_{\Lambda,\lambda}\to \tilde{W}_{\tilde \lambda}$ and 
$W'_{\Lambda',\lambda'}\to \tilde W'_{\tilde \lambda'}$ 
under the above isomorphisms.
\end{enumerate}
Moreover, by \cite[Proposition~A.4]{Mathieu}, we can 
chose $\tilde\fg$ and $\tilde \fg'$ to be of type~$A$. 
By \eqref{property.1}, the properties in claim~\eqref{class3.1} 
and claim~\eqref{class3.2} are true if and only 
if they are true for $\cO_{\tilde\lambda}(\tilde\fg,\tilde\fb)$ 
and $\cO_{\tilde\lambda'}(\tilde\fg',\tilde\fb')$. 
Furthermore \eqref{property.2} and \eqref{property.3} imply that also the property in
claim~\eqref{class3.3} is true if and only if it holds for $W$ and $\tilde W.$ 
Therefore it suffices to prove the equivalences between the 
three statements restricted to the integral setting. So, henceforth
we can assume $\Lambda=\Lambda_{{\rm int}}$ and $\Lambda'=\Lambda'_{{\rm int}}$.

Claim~\eqref{class3.1} implies claim~\eqref{class3.2} by Lemma \ref{gradedcentre}.
Now assume that claim~\eqref{class3.2} holds. The results repeated in 
Subsection~\ref{sscentre} imply that there must be a graded isomorphism 
between $\mathtt{C}_\lambda(W)$ and $\mathtt{C}_{\lambda'}(W')$.  
Application of Lemma~\ref{lemCox} then proves that 
claim~\eqref{class3.2} implies claim~\eqref{class3.3}.
The fact that  claim~\eqref{class3.3} implies claim~\eqref{class3.1} is
an immediate consequence of Lemma~\ref{maindereqLem}.
\end{proof}


\section{Ringel duality}\label{secRingel}

In this section we study Ringel duality for arbitrary blocks in parabolic category $\cO$ for arbitrary reductive Lie algebras. We will obtain two version of the Ringel duality functor, which generalise the ones in equations \eqref{eqInRi1} and \eqref{eqInRi2}.

\subsection{Ringel duality for parabolic category~$\cO$}

The main result here is:

\begin{theorem}\label{Ringeld}
Consider a reductive Lie algebra $\fg$ and $\lambda,\mu\in\intdom$.
\begin{enumerate}[$($i$)$]
\item\label{Ringeld.1} There are isomorphisms of 
(graded) algebras: $A^\mu_{\widehat\lambda}\;\cong\; R(A^\mu_\lambda)\;\cong\; A^{\widehat\mu}_\lambda$. 
\item\label{Ringeld.2} 
The restriction of the endofunctor $\cL_{l(w_0^\mu)}T_{w_0}$ of $\cO_\lambda$ 
to the full subcategory~$\cO^\mu_\lambda$ yields a Ringel duality functor 
$\underline\cR_\lambda^\mu$ satisfying
\begin{displaymath}
\underline\cR_\lambda^\mu:\,\cO_\lambda^\mu\to\cO_\lambda^{\widehat\mu}
\quad\,\mbox{ and }\;\quad\underline\cR_\lambda^\mu
(\Delta^{\mu}(x\cdot\lambda))\,\cong\, \nabla^{\widehat\mu}
(w_0w_0^\mu xw_0^\lambda \cdot{\lambda}). 
\end{displaymath}
\end{enumerate}
\end{theorem}

\begin{remark}{\rm
This theorem easily leads to a proof of Theorem A. In particular, 
there are blocks $\cO_\lambda^\mu$ which are not equivalent to 
$\cO_\lambda^{\widehat\mu}$, as demonstrated explicitly in 
Subsection~\ref{example} for completeness. }
\end{remark}

Before proving Theorem~\ref{Ringeld}, we establish some preparatory lemmata.

\begin{lemma}\label{Rstand}
For $x\in X^\mu_\lambda$, consider $\Delta^\mu(x\cdot\lambda)^\bullet$ 
as an object in $\cD^b(\cO_\lambda)$. Then
\begin{equation*}
\label{twisteq}\cL T_{w_0} \Delta^\mu(x\cdot\lambda)^\bullet\,\cong\, 
\nabla^{\widehat \mu}(w_0 w_0^\mu  x w_0^\lambda \cdot\lambda)^\bullet [l(w_0^\mu)].
\end{equation*}
\end{lemma}

\begin{proof}
Using  \cite[Section~6.5]{MOS} and Proposition~\ref{RHresult}, one can 
consider the Koszul dual of (the graded version of) 
Corollary~\ref{shuffstandards}. Corollary~\ref{corisomclosed} and 
equation~\eqref{commtwtr} then allow us to interpret this as the proposed 
isomorphism. We also provide the sketch of a proof without the use of 
Koszul duality, following the proof 
of \cite[Proposition~4.4(1)]{prinjective}. Applying parabolic 
induction to the classical BGG resolution (see 
\cite[Chapter~6]{Humphreys} or \cite{Lepowsky}) for the Levi subalgebra 
of $\fq_\mu$, yields an exact complex
\begin{equation}\label{eqjun24}
0\to C_{l(w_0^\mu)}\to\cdots \to C_j \to C_{j-1}\to \cdots C_0 \to \Delta^\mu(x\cdot\lambda)\to 0,\;\mbox{ with} 
\end{equation}
\begin{displaymath}
C_j=\bigoplus_{u\in W_\mu, l(u)=j}\Delta(u x\cdot\lambda).
\end{displaymath}
As Verma modules are acyclic for $T_{w_0}$, see \cite[Theorem~2.2]{AS}, 
we can compute $\cL T_{w_0} \Delta^\mu(x\cdot\lambda)^\bullet$ by evaluating  
$T_{w_0}$ at the complex \eqref{eqjun24} (where $\Delta^\mu(x\cdot\lambda)$ is deleted). 
Now \cite[Theorem~2.3]{AS} implies
\begin{displaymath}
T_{w_0} \Delta(u x\cdot\lambda)\cong \nabla(w_0 u x \cdot\lambda).
\end{displaymath}
The complex $D_\bullet =\cL T_{w_0} C_\bullet=T_{w_0} C_\bullet$ is hence of the form
\begin{displaymath}
D_j = \bigoplus_{u\in W_{\widehat\mu}, l(u)=j}
\nabla(u w_0 x\cdot\lambda)= \bigoplus_{u\in W_{\widehat\mu}, 
l(u)=l(w_0^\mu)-j}\nabla(u w_0 w_0^\mu x\cdot\lambda). 
\end{displaymath}
We have $w_0 w_0^\mu x\cdot\lambda=y\cdot\lambda$ with 
$y:=w_0w_0^\mu x w_0^\lambda\in X_\lambda^{\widehat\mu}$.
As the maps in the complex satisfy a dual version of the ones 
in the complex $C_\bullet$, we find
\begin{displaymath}
\cL T_{w_0} \Delta^\mu(x\cdot\lambda)^\bullet\;\cong \;
\nabla^{\widehat\mu}(y \cdot\lambda)^\bullet[l(w_0^\mu)], 
\end{displaymath}
which concludes the proof.
\end{proof}

\begin{lemma}\label{rightex}
The restriction of the endofunctor $\cL_{l(w_0^\mu)} T_{w_0}$ 
of~$\cO_\lambda$ to the full subcategory~$\cO_\lambda^\mu$ is right exact.
\end{lemma}

\begin{proof}
Lemma~\ref{Rstand} implies that $\cL_i T_{w_0} N=0$, if 
$i\not= l(w_0^\mu)$, for any $N\in\Ob(\cO_\lambda^\mu)$ 
with standard flag. So, in particular, for projective objects 
of~$\cO_\lambda^\mu$. The fact that $\cL_i T_{w_0} M=0$ for 
$i<l(w_0^\mu)$ and arbitrary $M\in\cO_\lambda^\mu$ can then 
be derived by induction on the projective dimension of~$M$ in 
$\cO_\lambda^\mu$, by considering short exact sequences.
\end{proof}

\begin{proof}[Proof of Theorem~\ref{Ringeld}]
First we prove part \eqref{Ringeld.2}. By Lemmata \ref{Rstand} and \ref{rightex}, the endofunctor 
$\cL_{l(w_0^\mu)} T_{w_0}$  of~$\cO_\lambda$ restricts to a right 
exact functor 
\begin{displaymath}
\underline\cR_\lambda^\mu\;:\;\cO^\mu_\lambda\;\to\;\cO_\lambda
\end{displaymath}
which maps  
standard modules to modules contained in the subcategory 
$\cO^{\widehat\mu}_\lambda$. Subsequently, all simple modules, being
the tops of standard modules, are mapped to modules in this 
subcategory. As $\cO^{\widehat\mu}_\lambda$ is a full Serre 
subcategory, this leads to the conclusion that we actually obtain a 
right exact functor
\begin{displaymath}
\underline\cR_\lambda^\mu\;:\;\cO^\mu_\lambda\;\to\;\cO_\lambda^{\widehat\mu}. 
\end{displaymath}
By Lemma \ref{Rstand}, this functor restricted to the full 
subcategory of modules in $\cO^\mu_\lambda$ with a standard flag, 
yields an exact functor from this category to the category of modules 
in $\cO_\lambda^{\widehat\mu}$ with costandard flag. This is an 
equivalence as it has as inverse, namely, the restriction of 
$\dd\cL_{l(w_0^\mu)}T_{w_0}\dd$, by Lemma \ref{Rstand}, equation 
\eqref{dertwist} and the fact that $\cO^\mu_\lambda$ is isomorphism 
closed in $\cO_\lambda$. Part (ii) then follows from 
\cite[Proposition~2.2]{prinjective}.

In particular, this implies $R(A_\lambda^\mu)\cong A_{\lambda}^{\widehat\mu}$.
By \cite[Theorem 11]{SoergelD} there is an isomorphism 
$A_\lambda\cong A_{\widehat\lambda}$. Under this isomorphism, the 
set of idempotents corresponding to $X^{\widehat\mu}_{\lambda}$ 
are mapped to the ones corresponding to $X^\mu_{\widehat\lambda}$, 
implying that $A_\lambda^{\widehat\mu}\cong A_{\widehat\lambda}^\mu$. 
This proves the ungraded isomorphisms of algebras in \eqref{Ringeld.1}. 

By the previous conclusions in this proof, \cite[Proposition~2.2(2)]{prinjective} implies that the functor 
$\cL_{l(w_0^\mu)}T_{w_0}$ maps a minimal projective generator of 
$\cO_\lambda^\mu$ to a characteristic q.h. tilting module of 
$\cO_{\lambda'}^\mu$. As the derived functor is an equivalence 
between the derived categories and is gradable, it follows that 
$R(A_\lambda^\mu)\cong A_\lambda^{\widehat\mu}$ holds as graded algebras.
\end{proof}

\begin{remark}{\rm
Moreover, \cite[Proposition 2.2]{prinjective} also implies that 
\begin{displaymath}
\cD^b(\cO_\lambda^\mu)\;\cong\;\cD^b(\cO_\lambda^{\widehat\mu})
\;\cong\;\cD^b(\cO_{\widehat\lambda}^\mu).
\end{displaymath}
When $\fg$ is not of type $A$, this is outside the scope of Theorem \ref{maindereq}. However it can be proved similarly 
as will be done in the next subsection.}
\end{remark}

\begin{remark}{\rm
An anonymous referee brought to our attention an insightful different method for proving 
Theorem~\ref{Ringeld}\eqref{Ringeld.1}. We set $\fq_\mu=\fl\oplus\fw^+$, with 
$\fw^+$ the radical and $\fl$ the Levi subalgebra of the parabolic subalgebra $\fq_\mu$, 
leading to parabolic decomposition $\fg=\fw^-\oplus \fl\oplus\fw^+$. 
Consider $A=U(\fg)$, $B=U(\fw^+)$, $H=U(\fl)$ and $\overline{B}=U(\fw^-)$. 
This gives a triangular decomposition in the sense of \cite[Section~2.1]{GGOR}, when 
taking into account footnote 1.

For a complex Lie algebra $\fa$ such that $\dim(\fa)=d<\infty$, 
the extension groups $\Ext^\bullet_{U(\fa)}(\mC,U(\fa))$ 
can be calculated from the Chevalley-Eilenberg complex 
\begin{displaymath}
0\to U(\fa)\to \fa^\ast\otimes U(\fa)\to\cdots\to \bigwedge^j(\fa^\ast)\otimes 
U(\fa)\to\cdots\to \bigwedge^d(\fa^\ast)\otimes U(\fa)\to 0. 
\end{displaymath}
This can be re-interpreted as the  Koszul resolution of the trivial 
$\fa$-module $\mC$, shifted over $d$ positions;
\begin{displaymath}
0\to \bigwedge^d(\fa)\otimes U(\fa)\to \bigwedge^{d-1}(\fa)\otimes 
U(\fa)\to\cdots\to \bigwedge^{d-j}(\fa)\otimes U(\fa)\to\cdots\to  U(\fa)\to 0, 
\end{displaymath}
yielding $\dim\Ext_{U(\fa)}^i(\mC,U(\fa))=\delta_{i,d}$.
Hence $B$ is Gorenstein and one can apply \cite[Proposition~4.3]{GGOR}, 
which implies that the Ringel dual of $\cO^\mu_\lambda$ is the category 
of {\em right} $A$-modules, which are finitely generated, $H$-semisimple 
(and locally finite), $B$-locally nilpotent and which admit generalised 
central character $\chi_\lambda$. To interpret these as left 
$A=U(\fg)$-modules we can apply an anti-automorphism of $U(\fg)$. 
The principal automorphism corresponding to $X\mapsto -X$ for $X\in \fg$ 
shows that the Ringel dual is equivalent to $\cO^\mu_{\widehat\lambda}$. 
The Chevalley anti-automorphism of $U(\fg)$ gives an equivalence with $\cO_\lambda^{\widehat\mu}$.
}
\end{remark}

\subsection{An alternative approach}

The following result is a reformulation of the results in the 
previous subsection, but we provide an alternative proof, 
stressing the link with our construction of derived equivalences.

\begin{proposition}\label{Ringel2}
We have $R(A_\lambda^\mu)\cong A_{\widehat\lambda}^\mu$. Moreover, 
the Ringel duality functor 
$\overline\cR_\lambda^\mu:\cO_\lambda^\mu\to\cO_{\widehat\lambda}^\mu$ 
satisfies, for every $x\in X_\lambda^\mu$, the following:
\begin{displaymath}
\overline\cR_\lambda^\mu(\Delta^\mu(x\cdot\lambda))\,\cong\, 
\nabla^\mu(w_0^\mu xw_0^\lambda w_0\cdot\widehat{\lambda}). 
\end{displaymath}
\end{proposition}

\begin{lemma}\label{lemderforR}
For all integral dominant $\lambda,\mu$, there is an equivalence of 
triangulated categories 
$F^\mu_\lambda:\cD^b(\cO_\lambda^\mu)\to \cD^b(\cO_{\widehat\lambda}^\mu)$ 
such that the following diagram commutes:
\begin{displaymath}
\xymatrix{
\cD^b(\cO_0)\ar[rr]^{\cL C_{w_0}}&&\cD^b(\cO_0)\\
\cD^b(\cO_{\lambda})\ar[rr]^{F_\lambda}\ar[u]^{\theta_\lambda^{out}}&&
\cD^b(\cO_{\widehat\lambda})\ar[u]^{\theta_{\widehat\lambda}^{out}} \\
\cD^b(\cO_{\lambda}^\mu)\ar[rr]^{F^\mu_\lambda}\ar[u]^{\imath^\mu}&&
\cD^b(\cO^\mu_{\widehat\lambda})\ar[u]^{\imath^\mu}. \\
}
\end{displaymath}
\end{lemma}

\begin{proof}
This is proved identically to Lemmata~\ref{equivlambda} and~\ref{equivpar}, 
the only change being that Proposition~\ref{propcommu}\eqref{propcommu.2} is used rather 
than Proposition~\ref{propcommu}\eqref{propcommu.1}.
\end{proof}

\begin{corollary}
For all $x\in X^\mu_\lambda$, we have
$F^\mu_\lambda(P^\mu(x\cdot\lambda)^\bullet)\;\cong 
\; T^\mu(w_0^\mu xw_0^\lambda w_0 \widehat\lambda)^\bullet[l(w_0^\mu)]$,
implying that  $R(A_{\widehat\lambda}^\mu):=\End_{\cO_{\widehat\lambda}^\mu}
(T^\mu_{\widehat\lambda})\cong A_\lambda^\mu$. 
\end{corollary}

\begin{proof}
Lemma \ref{lemderforR}, Proposition \ref{propcommu}\eqref{propcommu.2} 
and Corollary \ref{shuffstandards} imply
\begin{displaymath}
\theta_\lambda^{out}\imath^\mu F^\mu_\lambda 
P^\mu(x\cdot\lambda)^\bullet\cong \cL C_{w_0}\theta_x 
\Delta^\mu(e)^\bullet\cong \theta_{w_0 x w_0}L(w_0^\mu w_0)^\bullet[l(w_0^\mu)]. 
\end{displaymath}
Proposition \ref{tiltingtheta} then implies
\begin{displaymath}
\theta_\lambda^{out}\imath^\mu F^\mu_\lambda P^\mu(x\cdot\lambda)^\bullet\cong 
T^\mu(w_0^\mu x w_0)^\bullet[l(w_0^\mu)]\cong \theta^{out}_{\widehat\lambda}
\imath^\mu T^\mu(w_0^\mu x w_0w_0^{\widehat\lambda}\cdot\widehat\lambda)^\bullet[l(w_0^\mu)]. 
\end{displaymath}
The equivalence thus follows from Corollary~\ref{corisomclosed}.

The isomorphism of the endomorphism algebras then follows from the 
fact that both modules correspond to complexes contained in one 
position exchanged by an equivalence of triangulated categories.
\end{proof}

The proof of Proposition \ref{Ringel2} then follows easily from this corollary.

\subsection{Parabolic and singular Koszul-Ringel duality}

In this subsection we compose the 
Koszul duality \eqref{KoszuleqO} with the Ringel duality.
We also study the link with the category of linear complexes of q.h. tilting modules
$\mathfrak{LT}_\lambda^\mu$, a full subcategory of~$\cD^b({}^{\mZ}\cO^\mu_\lambda)$, see~\cite{pairing}. 
By \cite[Corollary~6]{pairing}, the 
category $\mathfrak{LT}_\lambda^\mu$ is equivalent to gmod-$R(A^\mu_\lambda)^{!}$. 
Hence Theorem~\ref{Ringeld} and equation \eqref{eqeps} implies that there are equivalences of categories
$$
\mathfrak{LP}^{\widehat{\mu}}_\lambda\;\cong\;{}^{\mZ}\cO^\lambda_\mu\;\cong\;\mathfrak{LT}_\lambda^\mu\;\cong\; 
{}^{\mZ}\cO^{\widehat\lambda}_{\widehat\mu}\;\cong\; \mathfrak{LP}^\mu_{\widehat\lambda}.
$$
We now give explicit descriptions of these equivalences, which will be useful
for practical application in e.g. \cite{CM2}.

\begin{lemma}\label{lemKR}
Consider integral dominant $\lambda,\mu$.
\begin{enumerate}[$($i$)$]
\item\label{lemKR.1} The functor $\cL \overline\cR_\lambda^\mu$ 
induces an equivalence of categories 
\begin{displaymath}
\cL\overline{\cR}^\mu_\lambda:\;\,\cD^b({}^{\mZ}\cO_\lambda^\mu)\;\,
\tilde{\to}\,\; \cD^b({}^{\mZ}\cO_{\widehat\lambda}^\mu), 
\end{displaymath}
which restricts to an equivalence of categories 
\begin{displaymath}
\mathfrak{LP}^\mu_{\lambda}\;\,\tilde{\to}\;\,
\mathfrak{LT}^\mu_{\widehat\lambda}\quad\mbox{ with }\quad 
P^\mu(x\cdot\lambda)^\bullet\;\mapsto\;
T^\mu(w_0^\mu x w_0^\lambda w_0 \cdot\widehat\lambda)^\bullet.
\end{displaymath}
\item\label{lemKR.2} 
The functor $\cL \underline\cR_\lambda^\mu$ induces an equivalence of categories 
\begin{displaymath}
\cL\underline{\cR}^\mu_\lambda:\;\,\cD^b({}^{\mZ}\cO_\lambda^\mu)\;\,
\tilde{\to}\,\; \cD^b({}^{\mZ}\cO_{\lambda}^{\widehat\mu}), 
\end{displaymath}
which restricts to an equivalence of categories 
\begin{displaymath}
\mathfrak{LP}^\mu_{\lambda}\;\,\tilde{\to}\;\,
\mathfrak{LT}^{\widehat\mu}_{\lambda}\quad\mbox{ with }\quad 
P^\mu(x\cdot\lambda)^\bullet\;\mapsto\;T^{\widehat\mu}(w_0w_0^\mu x w_0^\lambda \cdot\lambda)^\bullet. 
\end{displaymath}
\end{enumerate}
\end{lemma}

\begin{proof}
The equivalences of derived categories (in the ungraded sense) 
follow from \cite[Proposition 2.2(2)]{prinjective} and 
Theorem~\ref{Ringeld}. The graded version then follows from 
application of Proposition~\ref{liftequiv}. Alternatively,
the (graded) equivalence can be proved as in 
Section~\ref{secConstr} by using Proposition~\ref{propcommu}\eqref{propcommu.2} 
rather than Corollary~\ref{corforequi}. It is a general feature that the Ringel duality functor maps 
indecomposable projective modules to indecomposable q.h. 
tilting modules, see e.g. \cite[Proposition~2.1(2)]{prinjective}. 
The explicit description then follows from the action on standard modules in Theorem~\ref{Ringeld}.
\end{proof}

We compose functors to obtain the contravariant Koszul-Ringel 
duality functor
\begin{displaymath}
\Phi_\lambda^\mu\,=\,\cK^{\widehat\mu}_\lambda \circ 
\left(\cL\underline{\cR}_\lambda^{\widehat\mu}\right)^{-1}\,:\, 
\;\cD^b({}^{\mZ}\cO_\lambda^\mu)\;\,\tilde{\to}\;\, \cD^b({}^{\mZ}\cO_\mu^\lambda), 
\end{displaymath}
yielding an equivalence of triangulated categories.

\begin{corollary}\label{tiltingdual}
The contravariant functor $\Phi_\lambda^\mu$ above restricts 
to an equivalence 
$\Phi_\lambda^\mu: \mathfrak{LT}_\lambda^\mu\to {}^{\mZ}\cO^{\lambda}_{{\mu}}$ 
which satisfies, for any $x\in X^\mu_\lambda$, the following:
\begin{enumerate}
\item\label{tiltingdual.1} 
$\Phi^\mu_\lambda(T^\mu(x\cdot\lambda)^\bullet)\cong 
L(w_0^\lambda x^{-1}w_0^\mu\cdot\mu)$.
\item\label{tiltingdual.2} The complex $\cT^\bullet_L$ in 
$\mathfrak{LT}^\mu_\lambda$, which is isomorphic to $L(x\cdot\lambda)^\bullet$ 
in~$\cD^b({}^{\mZ}\cO^\mu_\lambda)$, satisfies 
$\Phi^\mu_\lambda(\cT^\bullet_L)\cong T^\lambda(w_0^\lambda x^{-1}w_0^\mu\cdot\mu)$.
\item\label{tiltingdual.3} The complex $\cT^\bullet_\nabla$ 
in $\mathfrak{LT}^\mu_\lambda$, isomorphic to 
$\nabla^\mu(x\cdot\lambda)^\bullet$ in~$\cD^b({}^{\mZ}\cO^\mu_\lambda)$, satisfies $\Phi^\mu_\lambda(\cT^\bullet_\nabla)\cong\Delta^{\lambda}(w_0^\lambda x^{-1}w_0^\mu\cdot\mu)$.
\item\label{tiltingdual.4} The complex $\cT^\bullet_\Delta$ in 
$\mathfrak{LT}^\mu_\lambda$, isomorphic to 
$\Delta^\mu(x\cdot\lambda)^\bullet$ in~$\cD^b({}^{\mZ}\cO^\mu_\lambda)$,
satisfies 
$\Phi^\mu_\lambda(\cT^\bullet_\Delta)\cong\nabla^{\lambda}(w_0^\lambda x^{-1}w_0^\mu\cdot\mu)$.
\end{enumerate} 
\end{corollary}

\begin{proof}
The first property follows from combination of 
Lemma~\ref{lemKR}\eqref{lemKR.2} and the well-known property 
$\cK^{\widehat\mu}_\lambda(P(x\cdot\lambda)^\bullet)\cong L(x^{-1}w_0\cdot\mu)^\bullet$, 
see e.g. \cite[Theorem~3.11.1]{BGS}. The other three properties then follow 
from the first one and \cite[Theorem~9]{pairing}. Property \eqref{tiltingdual.3} 
also follows from the combination of Theorem~\ref{Ringeld}\eqref{Ringeld.2} 
with the property $\cK^\mu_\lambda(\Delta^\mu(x\cdot\lambda)^\bullet)\cong \Delta^\lambda(x^{-1}w_0 \cdot\widehat\mu)^\bullet,$ see \cite[Theorem~3.11.1]{BGS}.
\end{proof}

The following is a standard property for categories of linear complexes.
\begin{lemma}
Consider $\cM^\bullet\in \Ob{\mathfrak{LT}^\mu_\lambda}$. 
Then $T^\mu(x\cdot\lambda)$ is a submodule of~$\cM^k$ if and only if 
$L(w_0^\lambda x^{-1} w_0^\mu\cdot\mu)\langle k\rangle$ is a subquotient 
of~$\Phi^\lambda_\mu(\cM^\bullet)$.
\end{lemma}

\begin{proof}
The property is clearly true for complexes contained in one position as,
by construction,
\begin{displaymath}
\Phi^\mu_\lambda(T^\mu(x\cdot\lambda)^\bullet[k]\langle k\rangle)
\cong L(w_0^\lambda x^{-1}w_0^\mu\cdot\mu)\langle k\rangle.
\end{displaymath}
The full statement then follows by induction, using distinguished 
triangles in $\cD^b({}^{\mZ}\cO_\lambda^\mu)$ such that the cone 
belongs to $\mathfrak{LT}^\mu_\lambda$.
\end{proof}

\subsection{A block which is not Ringel self-dual}\label{example} 
Consider $\fg=\mathfrak{sl}(4)$, with $W=S_4$ with simple reflections $s_1,s_2,s_3$ and choose $\lambda,\mu\in\intdom$ such
that $w_0^\lambda=s_3$ and $w_0^\mu=s_1$. The $\Ext^1$-quiver of $\cO^\mu_\lambda$ and $\cO^{\widehat\mu}_\lambda$ can be 
obtained from Proposition~\ref{propquiv} and 
\cite[Appendix~A]{Stroppel2}. As we have a duality preserving isomorphism classes of simple modules,
we neglect the directions on the edges. We order the simple modules according to the Bruhat order. The $\Ext^1$-quiver of $\cO^\mu_\lambda$ is given by{\scriptsize
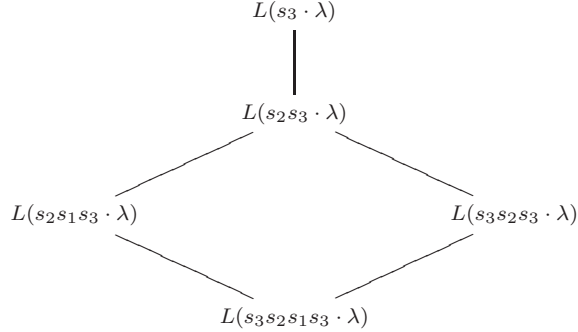
\begin{displaymath}
\xymatrix{
& L(s_3\cdot\lambda)\ar@{-}[d]\\
& L(s_2s_3\cdot\lambda)\ar@{-}[dl]\ar@{-}[dr]\\
L(s_2s_1s_3\cdot\lambda)\ar@{-}[dr]&&L(s_3s_2s_3\cdot\lambda)\ar@{-}[dl]\\ 
& L(s_3s_2s_1s_3\cdot\lambda)}
\end{displaymath}
}

We have $w_0^{\widehat\mu}=s_3$ and similarly we obtain the $\Ext^1$-quiver of $\cO^{\widehat\mu}_\lambda$,{\scriptsize
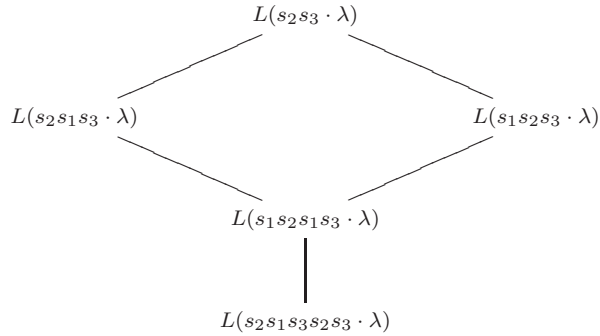
\begin{displaymath}
\xymatrix{
& L(s_2s_3\cdot\lambda)\ar@{-}[dl]\ar@{-}[dr]\\
L(s_2s_1s_3\cdot\lambda)\ar@{-}[dr]&&L(s_1s_2s_3\cdot\lambda)\ar@{-}[dl]\\ 
& L(s_1s_2s_1s_3\cdot\lambda)\ar@{-}[d]\\
& L(s_2s_1s_3s_2s_3\cdot\lambda)}
\end{displaymath}}

This implies that $\cO_\lambda^\mu$ 
contains precisely one simple module which has a first extension 
with only one other simple module. Its projective cover is a 
non-simple standard module, so is not injective. Also
$\cO^{\widehat\mu}_\lambda$ has precisely one simple module 
which has a first extension with only one other simple module. 
This module is a simple standard module, which thus has an 
injective projective cover by \cite{Irving85}. This implies 
that $\cO^\mu_\lambda$ is not equivalent to $ \cO^{\widehat\mu}_\lambda$.

\noindent
{\bf Acknowledgment.}
KC is a Postdoctoral Fellow of the Research Foundation - Flanders (FWO).
VM is partially supported by the Swedish Research Council,
Knut and Alice Wallenbergs Stiftelse and the Royal Swedish Academy of Sciences.

\noindent
KC: Department of Mathematical Analysis, Ghent University, Krijgslaan 281, 9000 Gent, Belgium;
E-mail: {\tt Kevin.Coulembier@Sydney.edu.au} 
\vspace{2mm}

\noindent
VM: Department of Mathematics, University of Uppsala, Box 480, SE-75106, Uppsala, Sweden;
E-mail: {\tt  mazor@math.uu.se}
\date{}

\end{document}